\documentclass[11pt, a4paper]{amsart}
\usepackage[latin1]{inputenc}
\usepackage{amsmath, amsfonts, amssymb, amsthm, mathtools, mathrsfs,xypic,bm}
\usepackage{enumerate}
\usepackage[numbers]{natbib}
\usepackage{setspace}
\usepackage[usenames,dvipsnames]{xcolor}
\usepackage{xcolor}
\usepackage{tikz-cd}

\usepackage[pagebackref,colorlinks,linkcolor=blue,citecolor=blue,urlcolor=blue,hypertexnames=true]{hyperref}
\usepackage[pagebackref,hypertexnames=true]{hyperref}

\newtheorem{theorem}{Theorem}[section]
\newtheorem{lemma}[theorem]{Lemma}
\newtheorem{proposition}[theorem]{Proposition}
\newtheorem{corollary}[theorem]{Corollary}

\theoremstyle{definition}
\newtheorem{definition}[theorem]{Definition}

\newtheorem{remark}[theorem]{Remark}

\newcommand{\ZZ}{\mathcal{Z}}
\newcommand{\OO}{\mathcal{O}}
\newcommand{\QQ}{\mathcal{Q}}
\newcommand{\N}{\mathbb{N}}
\newcommand{\Z}{\mathbb{Z}}
\newcommand{\Q}{\mathbb{Q}}
\newcommand{\R}{\mathbb{R}}
\newcommand{\C}{\mathbb{C}}
\newcommand{\mI}{\mathcal{I}}
\newcommand{\Sgr}{\mathcal{S}}
\newcommand{\spS}{\mathbb{S}}

\newcommand{\K}{\mathcal{K}}	
\newcommand{\Tor}{\mathrm{Tor}}
\newcommand{\Aut}{\mathrm{Aut}}
\newcommand{\Autgr}{\mathrm{Aut}_{\text{gr}}}

\newcommand{\Cliff}[1]{\C \ell_{#1}}
\newcommand{\id}[1]{\mathrm{id}_{#1}}
\newcommand{\Ggr}{\operatorname{G}^{\text{gr}}}
\newcommand{\G}{\operatorname{G}}
\newcommand{\KD}{K_0(D)^{\times}_{+}}
\newcommand{\hE}{\hat{E}}
\newcommand{\bE}{\bar{E}}

\DeclareMathOperator{\hocolim}{hocolim}

\definecolor{darkpastelred}{rgb}{0.76, 0.23, 0.13}
\definecolor{darkred}{rgb}{0.55, 0.0, 0.0}
\definecolor{darkmagenta}{rgb}{0.55, 0.0, 0.55}
\definecolor{coolblack}{rgb}{0.0, 0.18, 0.39}
\definecolor{ceruleanblue}{rgb}{0.16, 0.32, 0.75}

\begin{document}
\title[Draft]{Bundles of strongly self-absorbing $C^*$-algebras \\ with a Clifford grading}
\author{Marius Dadarlat}
\address{Department of Mathematics \\
Purdue University\\
West Lafayette, IN 47907, USA}
\author{Ulrich Pennig}
\address{School of Mathematics \\ 
Cardiff University \\
Cardiff \\
CF24 4AG\\
Wales\\
UK}

\maketitle
\begin{abstract}
We extend our previous results on generalized Dixmier-Douady theory to graded $C^*$-algebras,
as  means for explicit computations of the invariants arising for bundles of  ungraded $C^*$-algebras.
For a strongly self-absorbing $C^*$-algebra $D$ and  complex Clifford algebras $\mathbb{C}\ell_{n}$ we show that the classifying spaces of the groups of graded automorphisms $\mathrm{Aut}_{\text{gr}}(\mathbb{C}\ell_{n}\otimes \mathcal{K }\otimes D)$ admit compatible 
infinite loop space structures giving rise to a cohomology theory $\hat{E}^*_D(X)$. For $D$  stably finite and $X$  a finite CW-complex, we show that the tensor product operation defines a group structure on the isomorphism classes of  locally trivial bundles  of graded $C^*$-algebras  with fibers  $ \mathbb{C}\ell_{k}\otimes D \otimes \mathcal{K}$ and that this group is isomorphic to  $H^0(X,\mathbb{Z}/2)\oplus \hat{E}^1_{D}(X)$. 
Moreover, we establish  isomorphisms $\hat{E}^1_{D}(X)\cong   H^1(X;\mathbb{Z}/2) \times_{_{tw}} E^1_{D}(X)$ and
$\hat{E}^1_{D}(X)\cong E^1_{D\otimes \mathcal{O}_\infty}(X)$,
where $E^1_{D}(X)$ is the group that classifies the locally trivial bundles with fibers $D\otimes \mathcal{K}$.
In particular $E^1_{\mathcal{O}_\infty}(X)\cong H^1(X;\mathbb{Z}/2) \times_{_{tw}} E^1_{\mathcal{Z}}(X)$ where $\mathcal{Z}$ is the Jiang-Su algebra
and the multiplication on the last two factors is twisted  similarly to the Brauer theory for bundles with fibers the graded compact operators on a finite and respectively infinite dimensional Hilbert space.
\end{abstract}

\tableofcontents

\section{Introduction}
Continuous fields of $C^*$-algebras
 play the role of bundles of $C^*$-algebras, in the sense of topology,  as explained in \cite{BK}.
  These structures occur naturally in various generalizations of the Gelfand-Naimark theorem.  Indeed, Fell \cite{Fell} showed that any  separable $C^*$-algebra $A$ with Hausdorff primitive spectrum $X$
 has a canonical continuous field structure over $X$ with fibers the primitive quotients of $A$. Equally important, continuous field $C^*$-algebras are employed
as versatile tools in several areas, including index and representation theory, the Novikov and the Baum-Connes conjectures, strict deformation quantization, quantum groups and E-theory.
While in general the bundle structure that underlies a continuous field of $C^*$-algebras is typically not locally trivial, in this paper we are concerned with locally trivial bundles, {which we will call $C^*$-bundles.}

Strongly self-absorbing $C^*$-algebras \cite{paper:TomsWinter}
are separable unital $C^*$-algebras $D$ defined by a crucial property that they share with the complex numbers $\C$.  Namely, there exists an isomorphism $D\to D\otimes D$
which is unitarily homotopic to the map $d\mapsto d\otimes 1_D$ \cite{Dadarlat-Winter:KK-of-ssa}, \cite{paper:WinterZStable}.
Any strongly self-absorbing $C^*$-algebra $D$ is either stably finite or purely infinite. The latter condition is equivalent to $D\cong D\otimes \OO_\infty$, where  $\OO_\infty$ is the infinite Cuntz algebra.
Due to recent progress in classification theory \cite{Winter:abel} we now have a complete list of all the {strongly} self-absorbing $C^*$-algebras that satisfy the Universal Coefficient Theorem (abbreviated~UCT) in KK-theory, see Subsec.~\ref{subsec:ssa}.

In a series of papers \cite{DP3,DP2,DP1}, we have extended Dixmier-Douady theory \cite{paper:DixmierDouady} and the complex Brauer group \cite{paper:Grothendieck} from $C^*$-bundles with fibers Morita equivalent to $\C$ to $C^*$-bundles with fibers Morita equivalent to a self-absorbing $C^*$-algebra $D$.
Just like in the case of compact operators $\K$ on an infinite dimensional separable Hilbert space when the classifying space  $B\Aut(\K)$  is a model for the Eilenberg-McLane space $K(\Z,3)$,  it turns out that $B\Aut(D \otimes \K)$ is an infinite loop space. This not only implies that the set of isomorphism classes of $D\otimes \K$-bundles $[X,B\Aut(D \otimes \K)]=:E^1_D(X) $ has an abelian group structure as it corresponds to  the 1-group of a generalized cohomology theory $E^*_D(X)$, but also that its study is amenable to methods from stable homotopy theory.
Remarkably, the group law on $E^1_D(X)$ coming from the infinite loop space structure of  $B\Aut(D \otimes \K)$ coincides with the operation induced by the tensor product of $D \otimes \K$-bundles. Moreover,  similarly to the scalar case, the Brauer group associated to bundles with fibers $M_n(\C)\otimes D$, $n\geq 1$, is isomorphic to $\mathrm{Tor}\, \bar{E}^1_D(X)$, where $\bar{E}^1_D(X)$ is the subgroup of $E^1_D(X)$ corresponding to orientable $D \otimes \K$-bundles, \cite{DP3}.

Donovan and Karoubi studied bundles of graded algebras with fibers matrices over (complex) Clifford algebras $\Cliff{n}$ and computed the corresponding graded Brauer group,  \cite{paper:DonovanKaroubi}.
The infinite dimensional version of the graded Brauer group was computed in \cite{paper:Parker-Brauer} by
Parker for bundles with  fibers the graded compact operators $\K$.

In this paper we extend the  results from \cite{paper:DonovanKaroubi}  and \cite{paper:Parker-Brauer}
to bundles with fibers Morita equivalent to $D  \otimes \Cliff{n} $ for strongly self-absorbing $C^*$-algebras $D$.
We do not develop this direction for the sake of generalization alone but because we were led to it in
our attempts to understand the classification of $C^*$-bundles with fibers $O_\infty \otimes \K$
and more generally with fibers $D\otimes O_\infty \otimes \K$ as we are going to explain below.

Our major goal is to compute $E^1_D(X)$ for general strongly self-absorbing $C^*$-algebras~$D$. In addition to \cite{DP2,DP1}, this requires two more steps, the first of which is taken in this paper whereas the second step is completed in a follow-up paper which is joint work with Jim McClure \cite{DMcCP}. \emph{The first step} consists
in showing that in analogy with the ungraded case, the classifying space of the group of graded automorphisms $\Autgr(\Cliff{n} \otimes \K \otimes D)$ has an infinite loop space structure such that the associated generalized cohomology theory $\hE^*_D(X)$ has three attractive properties described in the following three theorems which are contained in Theorem~\ref{thm:neww},  Theorem~\ref{thm:basic}  and Corollary~\ref{Cor:crucial} in the text. The tensor product with $\id{\Cliff{1}}$ induces an equivalence of infinite loop spaces $\Autgr(\Cliff{n} \otimes \K \otimes D) \to \Autgr(\Cliff{n+1} \otimes \K \otimes D)$. Hence, these groups lead to the same cohomology theory for all $n \geq 1$. We assume that $X$ is a finite CW-complex and that $D$ is a stably finite strongly self-absorbing $C^*$-algebra satisfying the UCT.

\begin{theorem}\label{thm1:intro}
The tensor product operation defines a group structure on the isomorphism classes of  locally trivial bundles  of graded $C^*$-algebras $A$ with fibers  $A(x)\cong \Cliff{k(x)}\otimes D \otimes \mathcal{K}$, $x\in X$, $k(x)\geq 1$. This group is
isomorphic to  $H^0(X,\Z/2)\oplus \hE^1_{D}(X)$.
  \end{theorem}

\begin{theorem}\label{thm2:intro}
 There is an  isomorphism of groups
\[\hE^1_{D}(X)\cong   H^1(X;\Z/2) \times_{_{tw}} E^1_{D}(X)\]
with multiplication on the direct product
\[
	(w, \tau) \cdot (w',\tau') = (w + w', \tau + \tau' + j\circ\beta(w \cup w'))
\]
for $w,w' \in H^1(X,\Z/2)$ and $\tau,\tau' \in E^1_{D}(X)$, where $\beta \colon H^2(X,\Z/2) \to H^3(X,\Z)$ is the Bockstein homomorphism, $j \colon   E^1_{\C}(X)\to E^1_{D}(X)$ is
the map induced by the unital $*$-homomorphism $\C\to D$,   and we identify
$E^1_{\C}(X)\cong H^3(X,\Z)$.
\end{theorem}
\begin{theorem}\label{thm3:intro}
If $D\neq \C$, there is a natural isomorphism
$
		\hE_{D}^*(X)\cong E_{ D \otimes \OO_\infty}^*(X)
$
  \end{theorem}
From Theorems~\ref{thm2:intro} and ~\ref{thm3:intro} we deduce that that if $D\neq \C$ is stably finite, then
\[{E}^1_{D\otimes \OO_\infty}(X)\cong  H^1(X;\Z/2)\times_{_{tw}} E^1_{D}(X)\]

\emph{The second step} of our endeavor  is accomplished in  \cite{DMcCP} and concerns the calculation of $E^1_D(X)$ for $D$ stably finite. Altogether one obtains  a calculation for $E^1_{D}(X)$ whether $D$ is stably finite or purely infinite.

Let us preview the result from \cite{DMcCP} for the Jiang-Su algebra $\ZZ$ and for the infinite Cuntz algebra $\OO_\infty\cong \ZZ \otimes \OO_\infty$.
Let $k^*(X)$ denote the complex connective $K$-theory of the space $X$. For a finite based CW-complex $X$ with skeleta $X_i$, $\widetilde{k}^i(X)\cong \widetilde{K}^i(X,X_{i-2})$ and in particular $k^5(X)\cong K^1(X,X_3)$.
\begin{theorem}[\cite{DMcCP}]\label{thm:mainA} Let $X$ be a finite connected CW-complex. There are (not natural) isomorphisms
\begin{itemize}
\item[(a)] ${E}^1_{\ZZ}(X)\cong  H^3(X,\Z)\oplus k^5(X)$.
\item[(b)] ${E}^1_{\OO_\infty}(X)\cong \big(H^1(X, \Z/2)\times_{_{tw}} H^3(X,\Z)\big)\oplus k^5(X)$.
\end{itemize}
The multiplication on $ H^1(X;\Z/2) \times H^3(X,\Z)$ is given by
\[
	(w, \tau) \cdot (w',\tau') = (w + w', \tau + \tau' + \beta(w \cup w'))
\]
for $w,w' \in H^1(X,\Z/2)$ and $\tau,\tau' \in H^3(X,\Z)$, where $\beta \colon H^2(X,\Z/2) \to H^3(X,\Z)$ is the Bockstein homomorphism.
\end{theorem}
We refer the reader to \cite{DMcCP} for a computation of $E^1_{D}(X)$ and  $E^1_{{\OO_\infty}\otimes D}(X)$ for all stably finite strongly self-absorbing $C^*$-algebras  satisfying the UCT.

In the last part of the paper we show that the Brauer group $\hat{Br}_D(X)$ arising from graded bundles with fibers 
$M_n(D)\otimes \Cliff{k}$ is isomorphic to $$H^0(X,\Z/2) \times H^1(X,\Z/2)\times_{_{tw}} \Tor{\bar{E}^1_D(X)},$$
where $\bar{E}^1_D(X)\subset {E}^1_D(X)$ classifies $D\otimes K$-bundles with structure group $\Aut_0(D\otimes\K)$, the connected component of identity of $\Aut(D\otimes\K)$.

Let us give some background and describe in more detail the strategy taken in this paper.
The link between cohomology theories, spectra and infinite loop spaces is reviewed in Sec.~\ref{subsub-scth}.
We rely on work of Schlichtkrull ~\cite{paper:Schlichtkrull}.
Let $\text{Top}_*$ denote  the category of  based compactly generated Hausdorff spaces. Let $\mI$ be the category with objects $\mathbf{n} = \{1, \dots, n\}$ for $n \in \N_0$ (so including the empty set $\mathbf{0}$) and morphisms given by injective maps. This is a symmetric monoidal category where $\mathbf{m} \sqcup \mathbf{n} = \{1, \dots, m+n\}$ and $f \sqcup g$ acts by identifying the first $m$ entries of $\mathbf{m} \sqcup \mathbf{n}$ with $\mathbf{m}$ and the last $n$ entries with $\mathbf{n}$. The symmetry $\mathbf{m} \sqcup \mathbf{n} \to \mathbf{n} \sqcup \mathbf{m}$ is given by block permutation. An $\mI$-space is a functor $\mI \to \text{Top}_*$ and an $\mI$-monoid is a monoid object in the category of $\mI$-spaces (note that this makes use of the monoidal structure on $\mI$). The symmetry of $\mI$ can be used to define commutative $\mI$-monoids. These provide the input to one of the many infinite loop space machines in algebraic topology \cite{paper:MayThomason}. More precisely, any commutative $\mI$-monoid $X$ gives rise to a $\Gamma$-space  denoted by $\Gamma(X)$ and if the $\mI$-space is grouplike in the sense that the monoid $\pi_0(\hocolim_{\mI} X)$ is actually a group, then $X_{h\mI} = \hocolim_{\mI} X$ is an infinite loop space and the higher deloopings, ie.\ the spaces in the $\Omega$-spectrum are explicitly constructible from the $\Gamma$-space structure \cite{paper:SegalCatAndCoh}.
  
The category of spectra is a closed model category in the sense of Quillen. In particular, it comes with distinguished subsets of the morphisms called weak equivalences, fibrations and cofibrations generalising the corresponding concepts for topological spaces. Using this structure it is possible to formally invert the weak equivalences. The resulting category is called the stable homotopy category of spectra.  Using the Atiyah-Hirzebruch spectral sequence one verifies  that a weak equivalence of spectra induces a natural isomorphism of the cohomology theories on finite CW-complexes represented by them.  In \cite{BousfieldFriedlander} Bousfield and Friedlander developed a similar model category structure on $\Gamma$-objects in simplicial sets, such that its homotopy category is equivalent to the stable homotopy category of connective spectra \cite[Thm.~5.8]{BousfieldFriedlander}. This was later refined by Schwede in \cite{paper:SchwedeGamma}, who also extended it to $\Gamma$-spaces (i.e.\ $\Gamma$-objects in topological spaces) and proved it to be immaterial whether one considers topological spaces or simplicial sets \cite[Thm.~B1]{paper:SchwedeGamma}. While Schwede's model category structure has slightly different fibrations and cofibrations, it does have the same weak equivalences. Thus, its homotopy category is still equivalent to the one of connective spectra.

We will not use the full model category structure in this paper, but we will frequently adopt the following notation: If a map of $\Gamma$-spaces $\Gamma(X)\to \Gamma(Y)$ induces an equivalence in the stable homotopy category of spectra, we write  $\Gamma(X)\simeq \Gamma(Y)$. In this case $\Gamma(X)$ and $\Gamma(Y)$ are called stably weakly equivalent (see \cite{BousfieldFriedlander} or \cite[p.~331]{paper:SchwedeGamma}).

In this paper we focus on three commutative $\mI$-monoids $(\Omega^\infty KU^D)^*$,
$\G_D$, and  $\Ggr_D$,  associated to a strongly self-absorbing $C^*$-algebra, and which we describe in the sequel. In order to prove Theorem~\ref{thm3:intro}
we show first that {for stably finite $D$} there is an equivalence of $\Gamma$-spaces
 \begin{equation}\label{eq:333}
		 \Gamma(\Ggr_D)\simeq \Gamma(\Omega^\infty KU^D).
\end{equation}
We have shown in \cite{DP2} that
\begin{equation}\label{eq:444}
\Gamma(\G_{D\otimes \OO_\infty})\simeq \Gamma(\Omega^\infty KU^{D\otimes \OO_\infty}).
\end{equation}

From \eqref{eq:333} and  \eqref{eq:444}  we obtain that if $D\neq \C$ is a stably finite strongly self-absorbing $C^*$-algebra satisfying the UCT, then 
\begin{equation}\label{eq:222}
		\Gamma(\Ggr_D)\simeq \Gamma(\G_{D\otimes \OO_\infty})
\end{equation}
 since
$\Gamma(\Omega^\infty KU^D)\simeq \Gamma(\Omega^\infty KU^{D\otimes \OO_\infty})$ by \cite{DP2}.
Now Theorem~\ref{thm3:intro} follows from \eqref{eq:222} since $\hE_{D}^*(X)$ is the cohomology theory associated to $\Gamma(\Ggr_D)$
while $E_{ D \otimes \OO_\infty}^*(X)$ is associated to
$\Gamma(\G_{D\otimes \OO_\infty})$.

Let us discuss $(\Omega^\infty KU^D)^*$ first.
We have shown in \cite{DP2} that each strongly self-absorbing $C^*$-algebra~$D$ gives rise to a symmetric spectrum $KU^D$ with $n$th space
\[
	KU_n^D = \hom_{\text{gr}}(\Sgr, (\Cliff{1} \otimes D \otimes \K)^{\otimes n}),
\]
where $\Sgr = C_0(\R)$ is viewed as a graded $C^*$-algebra equipped with the grading by odd and even functions and $\Cliff{d}$ is the complex Clifford algebra. Moreover, the sequence of spaces $\Omega KU^D_1, KU^D_1, KU^D_2, \dots$ is an $\Omega$-spectrum, in which the maps $KU^D_n \to \Omega KU^D_{n+1}$ for $n \geq 1$ are induced by Bott periodicity. The associated cohomology theory is $X \mapsto K_*(C(X) \otimes D)$, \cite[Sec.~1.5]{Gue-Hig:book} (see also \cite[Constr.~6.4.9]{book:SchwedeGlobal}). 

In fact, since $D$ is strongly self-absorbing, these groups form a multiplicative cohomology theory. There is a comultiplication on $\Delta \colon \Sgr \to \Sgr \otimes \Sgr$ that can be used to give $KU^D$ the structure of a commutative symmetric ring spectrum implementing this multiplication. Just as a commutative unital ring~$S$ has a group of invertible elements $GL_1(S)$, a commutative symmetric ring spectrum $R$ has an associated spectrum of units $gl_1(R)$. The associated infinite loop space $GL_1(R)$ consists of those path-components of $\Omega^\infty R$ that are invertible in the ring~$\pi_0(R)$.

In~\cite{paper:Schlichtkrull} Schlichtkrull found a convenient way of obtaining the spectrum $gl_1(R)$ as a diagram spectrum.
The commutative $\mI$-monoid that gives the unit spectrum of $KU^D$ is defined as follows: First note that there is a unit element $\eta_m \in \Omega^m KU^D_m$ corresponding to the map $S^m \to KU^D_m$ that is part of the structure of a symmetric ring spectrum. Denote by $\mu_{m,n} \colon KU^D_m \wedge KU^D_n \to KU^D_{m+n}$ the ring spectrum multiplication. Let $(\Omega^n KU^D_n)^*$ be those elements that are stably invertible in the sense that there exists $g \in \Omega^m KU^D_m$ for some $m \in \N_0$ such that $\mu_{n,m} \circ (f \wedge g)$ and $\mu_{m,n} \circ (g \wedge f)$ are both homotopic to $\eta_{n+m} \in \Omega^{n+m} KU^D_{n+m}$. The sequence of spaces
\[
	(\Omega^\infty KU^D)^*(\mathbf{n}) = (\Omega^n KU^D_n)^*\
\]
extends to a commutative $\mI$-monoid that gives the unit spectrum $gl_1(KU^D)$ when plugged into the infinite loop space machine.

The sequence of groups $G_D(\mathbf{n}) = \Aut((D \otimes \K)^{\otimes n})$ equipped with the pointwise norm topology can be extended to a commutative $\mI$-monoid as well and therefore also gives rise to an $\Omega$-spectrum. We showed in \cite{DP2} that the associated cohomology theory $E_D^*(X)$ satisfies
\[
	E^0_D \cong [X, \Aut(D \otimes \K)] \quad , \quad E^1_D \cong [X, B\Aut(D \otimes \K)].
\]
For each $n \in \N$ the group $\Aut((D \otimes \K)^{\otimes n})$ acts on $KU^D_n$ by composition. This action induces a morphism of commutative $\mI$-monoids
\(
	G_D \to (\Omega^\infty KU^D)^*,
\)
which then gives rise to a map
\(
	B\Aut(D \otimes \K) \to BGL_1(KU^D)
\)
that turns out to be an isomorphism on all homotopy groups except for $\pi_1$, where it corresponds to the embedding $GL_1(K_0(D))_+ \to GL_1(K_0(D))$. Hence, the above map is an equivalence for purely infinite strongly self-absorbing $C^*$-algebras, since $GL_1(K_0(D))_+ = GL_1(K_0(D))$ in this case.

The $\mI$-monoid $G_D$ is constructed from automorphisms of the ungraded $C^*$-algebras $D \otimes \K$ and ignores the Clifford algebras that feature in the definition of $KU^D$. In this paper we will complete the picture by also considering
the commutative $\mI$-monoid $\Ggr_D$ of graded automorphisms given by
\[
	\Ggr_D(\mathbf{n}) = \Autgr((\Cliff{1} \otimes D \otimes \K)^{\otimes n})\ .
\]
As alluded to above, $\pi_0(GL_1(KU^D)) \cong GL_1(K_0(D))$, while $\pi_0(\Aut(D \otimes \K)) \cong GL_1(K_0(D))_+$. Because of the close link between the graded Clifford algebras and the additive group completion in $K$-theory, the expected effect of including them into the definition of $G_D$ is that $\pi_0$ should change from $GL_1(K_0(D))_+$ to $GL_1(K_0(D))$. As we will show in  Lem.~\ref{lem:Autgr_Cl1_short_exact} and Lem.~\ref{lem:Autgr_Cl2_equiv}, we have
	\[
		\pi_0(\Ggr_D(\mathbf{n})) \cong \Z/2 \times K_0(D)^\times_+\ .
	\]
In case $D$ is purely infinite we see that we are ``group-completing twice'' here. Nevertheless, the action on $KU^D$ remedies this and our main technical result (namely, the isomorphism  \eqref{eq:333}) can be summarized as follows (see Thm.~\ref{thm:action_on_KUD}):

\begin{theorem}
	The canonical action of the commutative $\mI$-monoid $\Ggr_D$ on the ring spectrum $KU^D$ gives a map of $\Gamma$-spaces
	$\Gamma(\Ggr_D) \to \Gamma(\Omega^\infty(KU^D)^*)$ which induces an isomorphism on all homotopy groups $\pi_n$ with $n > 0$ of the corresponding connective spectra and the homomorphism
		\(
		\{\pm 1\} \times K_0(D)_{+}^\times \to K_0(D)^\times, \,(a,b) \mapsto a \cdot b
	\)
	on $\pi_0$. In particular, it is an equivalence in the stable homotopy category of spectra if $D$ is stably finite and satisfies the UCT.
\end{theorem}

The article is structured as follows: Section~2 contains preliminary material about strongly self-absorbing $C^*$-algebras -- including a full list of all of the ones that satisfy the UCT -- followed by some background about graded $C^*$-algebras.
In Section~3 we study the homotopy type of the automorphism groups $\Autgr(\Cliff{n} \otimes A)$ for simple, trivially graded $C^*$-algebras $A$ (sometimes under the additional stability assumption that $A \otimes \K \cong A$).

In Section~4  we give some background on spectra  and  review the necessary details about commutative $\mI$-monoids and their relation to unit spectra.
We give details about the spectrum $KU^D$ and its unit spectrum. Then we highlight the relationship between $KU^{M_P}$ and the localisation $KU[P^{-1}]$.

Section~5 contains the first main result of this work.  In Subsect.~5.1 we give a summary of the previously obtained results on generalised Dixmier-Douady theory.
The commutative $\mI$-monoid $\Ggr_D$ is introduced in Sec.~5.2 and its stable homotopy type is studied in Lem.~\ref{lem:Ggr_stable} and Lem.~\ref{lem:comp_inv}. The comparison with the commutative $\mI$-monoid that represents the unit spectrum of $KU^D$ is then completed in Thm.~\ref{thm:action_on_KUD}. The resulting list of commutative $\mI$-monoids that represent the same cohomology theories is given in Cor.~\ref{Cor:crucial}. This corollary is crucial for the computation of $E_D^1(X)$ for stably finite $D$ discussed in Section~6. In Subsect.~\ref{The Brauer group} we compute the graded Brauer group, see Theorem~\ref{thm:Brauer_Serre} .

\section{$C^*$-algebras}

\subsection{Strongly self-absorbing $C^*$-algebras}\label{subsec:ssa}
The class of strongly self-absorbing $C^*$-algebras was introduced by Toms and Winter \cite{paper:TomsWinter}.
They are separable unital $C^*$-algebras $D$ singled out by the property that there exists an isomorphism $D\to D\otimes D$
which is unitarily homotopic to the map $d\mapsto d\otimes 1_D$ \cite{Dadarlat-Winter:KK-of-ssa}, \cite{paper:WinterZStable}.
Any strongly self-absorbing $C^*$-algebra is either stably finite or purely infinite.
Due to recent progress in classification theory \cite{Winter:abel} we now have a complete list of all the self-absorbing $C^*$-algebras that satisfy the Universal Coefficient Theorem (abbreviated~UCT) in KK-theory. We review this list below.

Recall that $\ZZ$ denotes the Jiang-Su algebra, $\OO_2$  the Cuntz algebra with two generators and $\OO_\infty$  the Cuntz algebra on infinitely many generators.
The $C^*$-algebra $\ZZ$ can be viewed as the infinite dimensional stably finite version of $\C$ whereas $\OO_\infty$ can be viewed as the purely infinite version. There are isomorphisms $\ZZ\otimes \C \cong \ZZ$ and $\OO_\infty \otimes \ZZ \cong \OO_\infty$.

The unital $*$-homomorphisms $\C \to \ZZ \to \OO_\infty$ are $KK$-equivalences. $\OO_2$ is KK-contractible.
All  UCT strongly self-absorbing $C^*$-algebras with the exception of $\C$ and $\OO_2$ are obtained from either $\ZZ$ or  $\OO_\infty$ via a construction similar to localization at a set of primes.

For a prime $p$, we let $ M_{p}$ denote the infinite tensor product
$ M_{p} =M_p(\C)^{\otimes \infty}$. If $P$ is a set of primes, then
$M_P$ is defined as the tensor product
\[ M_P=\bigotimes_{p\in P} M_{p}.\]
We adopt the following convention:
if   $P=\emptyset$, then $M_P=\C$.
Any UCT stably finite self-absorbing $C^*$-algebra is isomorphic to either:
$\C$, $\ZZ$, or $M_P$ for some nonempty set $P$ of primes.
Any purely infinite UCT self-absorbing $C^*$-algebra is isomorphic to either:
$\OO_\infty$,  or $M_P\otimes \OO_\infty$ for some nonempty set $P$ of primes or to $\OO_2$.

It is known that $M_P\otimes \ZZ \cong M_P$. More generally if $D$ is strongly self-absorbing and $D\neq \C$, then $D\cong D\otimes  \ZZ$. Note that that if $P\subset P'$, then $M_{P'}\otimes M_{P}\cong M_{P'}$.
If  $P$ is the set of all primes, we denote $M_P$ by $\QQ$. The following diagram is illustrative of the relationships between strongly self-absorbing $C^*$-algebras.
\[
\begin{tikzcd}[row sep=0.4cm]
{} &\ZZ\ar[dd] \ar[r]  & M_P \ar[dd]\ar[r] & {\QQ}\ar[dd]\ar[dr] &{}\\
{\C}\ar[ur]\ar[dr] &{} & {}  & {}&{\OO_2}\\
{} &\OO_\infty \ar[r]  & \OO_\infty \otimes M_P\ar[r] & {\OO_\infty \otimes \QQ}\ar[ur] &{}	
\end{tikzcd}
\]
An arrow $D \to D'$ in the diagram above indicates the property that $D'\otimes D \cong D'$. If $P$ is the set of all primes different from a fixed prime $p$, the corresponding algebra $M_P$ is denoted by $M_{(p)}$.

If $D$ is any UCT strongly self-absorbing $C^*$-algebra, then  $K_1(D)=0$ and
 $K_0(D)\cong K_0(\OO_\infty \otimes D)$ has a natural unital commutative ring structure with multiplication
 induced by the isomorphism $D\otimes D \cong D$. Let $\Z_p=\Z[\frac{1}{p}]$ denote the localization of $\Z$ away from $p$
and  $\Z_{(p)}$ the localization of $\Z$ at $p$.
 We have
 \[K_0(\C)\cong K_0(\ZZ)\cong K(\OO_\infty)\cong \Z.\]
  \[K_0(M_P)\cong  K_0(M_P\otimes \OO_\infty)  \cong\bigotimes_{p\in P}\Z_p=:\Z_P,\]
  \[K_0(\OO_2)=0.\]

Thus:
 $K_0(M_{p})\cong \Z_p$,
$K_0(M_{(p)})\cong \Z_{(p)}$, and $K_0(\QQ)\cong \Q$.
If $D$ is purely infinite $D\otimes \OO_\infty \cong D$.
The invertible elements of the commutative ring $K_0(D)$ are denoted by $K_0(D)^{\times}=GL_1(K_0(D))$. The subgroup of positive elements of  $K_0(D)^{\times}$ corresponding to classes of idempotents in $D\otimes \K$ is denoted by $K_0(D)^{\times}_{+}$. If $D$ is in addition purely infinite then $K_0(D)^{\times}_{+}=K_0(D)^{\times}$.

\subsection{Graded $C^*$-algebras}
We recall some basic points about graded $C^*$-algebras, mostly to fix notation: A \emph{grading} on a $C^*$-algebra $A$ is an automorphism $\gamma \in \Aut(A)$, called the \emph{grading automorphism}, with $\gamma^2 = \id{A}$. We define
\[
	A^{\rm ev} = \{ a\in A \ | \ \gamma(a) = a \} \qquad \text{and} \qquad A^{\rm odd} = \{ a\in A \ | \ \gamma(a) = -a \}
\]
and note that both are closed, linear and self-adjoint subspaces of $A$ such that $A = A^{\rm ev} \oplus A^{\rm odd}$ as Banach spaces. Moreover, $A^{\rm ev}$ is a $C^*$-subalgebra of $A$ and $A^{\rm odd} \cdot A^{\rm odd} \subseteq A^{\rm ev}$, $A^{\rm odd} \cdot A^{\rm ev} \subseteq A^{\rm odd}$ and $A^{\rm ev} \cdot A^{\rm odd} \subseteq A^{\rm odd}$. If $\gamma = \id{A}$, then $A^{\rm ev} = A$, $A^{\rm odd} = 0$ and we say that $A$ is \emph{trivially graded}. We say that elements of $A^{\rm ev}$ have degree $0$ and write $\deg(a) = 0$ for $a \in A^{\rm ev}$, likewise elements of $A^{\rm odd}$ have degree $1$ and we write $\deg(a) = 1$ for $a \in A^{\rm odd}$. These elements are called \emph{homogeneous}. Recall that \emph{graded tensor products} of two graded $C^*$-algebras $A$ and $B$ are obtained as completions of the algebraic tensor product $A \odot B$ equipped with the multiplication and involution defined on homogeneous elements by
\[
	(a \otimes b) \cdot (a' \otimes b') = (-1)^{\deg(b)\cdot \deg(a')}\,aa' \otimes bb' \qquad \text{and} \qquad (a \otimes b)^* = (-1)^{\deg(a)\cdot \deg(b)} a^* \otimes b^*
\]
for $a \in A$ and $b \in B$. We will use the symbol $\otimes$ for the minimal graded tensor product. The tensor flip in graded $C^*$-algebras takes the following form on homogeneous elements:
\begin{equation} \label{eqn:tensor_flip}
	\epsilon_{A,B} \colon A \otimes B \to B \otimes A \qquad , \qquad a \otimes b \mapsto (-1)^{\deg(a)\deg(b)} b \otimes a
\end{equation}
\subsubsection{Standard even grading}\label{even grading} If $A$ is any ungraded $C^*$-algebra, the standard even grading of $M_2(A)\cong M_2(\C)\otimes A$ is induced by the diagonal/off-diagonal grading of $M_2(\C)$. By identifying the compact operators $\K$ with $M_2(\K)$ one obtains the standard even grading of $\K$ (the choices made here are inconsequential) \cite{Bla:k-theory}.
\subsubsection{Clifford algebras}
One of the main example of graded $C^*$-algebras that we need are the Clifford algebras $\Cliff{n}$. Let $V = \C^n$ and consider the tensor algebra $T(V) = \bigoplus_{k=0}^\infty V^{\otimes k}$ modulo the ideal $I(V)$ generated by the relation $v \cdot v - \langle v, v \rangle 1 = 0$. Then $\Cliff{n} = T(V)/I(V)$. These are finite-dimensional algebras and if $\{e_1, \dots, e_n\}$ denotes the standard basis of $\C^n$, then the following relations hold in $\Cliff{n}$:
\[
	e_i^2 = 1 \qquad \text{and} \qquad e_i e_j + e_j e_i = 0 \text{ for } i \neq j\ .
\]
The algebra $\Cliff{n}$ carries a natural grading, with respect to which the vectors $v \in V$ are odd elements. It also carries a natural involution, such that the standard basis vectors are self-adjoint. These structures turn $\Cliff{n}$ into a graded $C^*$-algebra, \cite{Kas:KK}.

By the classification of complex Clifford algebras (see for example \cite[Thm.~4.3]{LawsonMichelsohn})  there are graded isomorphisms
\begin{equation} \label{eqn:Cliff_periodicity}
	\Cliff{2n} \cong \Cliff{2} \otimes M_{2^{n-1}}(\C) \qquad \text{and} \qquad \Cliff{2n+1} \cong \Cliff{1} \otimes M_{2^{n}}(\C)
\end{equation}
where the matrix algebra $M_{2^{n-1}}(\C)$ is trivially graded.

Let $A$ be a trivially graded $C^*$-algebra.
Under the isomorphism $\Cliff{2}\otimes A \cong M_2(A)$ the grading induced by $\Cliff{2}$ corresponds the standard even grading of $M_2(A)$.

\subsubsection{The graded suspension} Another graded $C^*$-algebra that will play a crucial role in the construction of the symmetric spectrum representing $K$-theory in the next section is the algebra $\Sgr = C_0(\R)$ equipped with the grading by odd and even functions. As explained in \cite[p.~94]{Joachim_HigherCoherences} the algebra $\Sgr$ can be equipped with a coassociative and cocommutative comultiplication
\(
	\Delta \colon \Sgr \to \Sgr \otimes \Sgr
\)
and a counit $\epsilon \colon \Sgr \to \C$, where $\Sgr \otimes \Sgr$ is the graded tensor product. To fix $\Delta$ note that by \cite[\S 3]{paper:Trout} the $*$-homomorphisms $\Sgr \to A$ correspond bijectively to odd, self-adjoint, regular unbounded multipliers of $A$. If $X$ denotes the identity function of $\R$, then $\Delta$ is the $*$-homomorphism corresponding to the unbounded multiplier $1 \otimes X + X \otimes 1$. The counit is easier to define and is given by $\epsilon(f) = f(0)$.
\section{Automorphisms of $C^*$-algebras with a Clifford grading}

In this section we study the homotopy type of the groups $\Autgr(\Cliff{n} \otimes A)$, ie.\ the automorphism groups of graded $C^*$-algebras $\Cliff{n} \otimes A$, where $A$ is a trivially graded simple $C^*$-algebra and $\Cliff{n}$ carries its natural $\Z/2\Z$-grading.

The periodicity of complex Clifford algebras \eqref{eqn:Cliff_periodicity} reduces this question to understanding $\Autgr(\Cliff{1} \otimes A)$ and $\Autgr(\Cliff{2} \otimes A)$. If we denote the even and odd parts of $\Cliff{n}$ by $\Cliff{n}^{\,\rm ev}$ and $\Cliff{n}^{\,\rm odd}$, respectively, then we have a (non-canonical) isomorphism of ungraded algebras
\begin{equation} \label{eqn:Cliff_even}
	\Cliff{n}^{\,\rm ev} \cong \Cliff{n-1}\ ,
\end{equation}

\begin{lemma} \label{lem:Autgr_Cl1_short_exact}
Let $A$ be a simple and trivially graded $C^*$-algebra.
 The grading automorphism $\nu$ is the generator of $\Autgr(\Cliff{1} )\cong \Z/ 2\Z$.
The group homomorphism
\[
	\psi \colon \Autgr(\Cliff{1} )\times \Aut(A) \to \Autgr(\Cliff{1} \otimes A)
		\quad , \quad (\beta ,\alpha)\mapsto \beta \otimes \alpha.
\]
is an isomorphism of topological groups.
\end{lemma}

\begin{proof} $\Cliff{1}$ is generated by $e_1$ with $e_1^2=1$. It follows that $\Cliff{1}$ has two proper
orthogonal central projections $c^+ = \frac{1 + e_1}{2}$ and $c^- = \frac{1 - e_1}{2}.$
One has $c^+ (x+ye_1)=(x+y)c^+$ and $c^{-}(x+ye_1)=(x-y)c^{-}$ for $x,y \in \C$.
The grading automorphism $\nu$ of $\Cliff{1}$ satisfies $\nu(x+ye_1)=x-ye_1.$
We identify the multiplier algebras
	\(
		M(\Cliff{1} \otimes A) \cong \Cliff{1} \otimes M(A)
	\).
The decomposition
\[\Cliff{1} \otimes A=(c^+ \otimes 1)(\Cliff{1} \otimes A )+(c^-\otimes 1)(\Cliff{1} \otimes A)\cong A \oplus A\]
gives an isomorphism of graded $C^*$-algebras
$\theta : \Cliff{1} \otimes A \to A \oplus A$, $\theta(1\otimes a+e_1\otimes b)=(a+b,a-b)$
under which the grading automorphism $\nu \otimes \id{A}$ corresponds to the map $\nu_A$ that flips the components of direct sum of $C^*$-algebras $A\oplus A$, $\nu_A(a,b)=(b,a)$. Since $A$ is simple, the primitive ideal space $\text{Prim}(A \oplus A)$ consists of two points. Thus  any automorphism $\gamma$  of $A \oplus A$ is either of the form $(a,b) \mapsto (\alpha(a), \beta(b))$ or  $(a,b) \mapsto (\alpha(b), \beta(a))$ for some $\alpha,\beta \in \Aut(A)$. Since any graded automorphism must commute with $\nu_A$ we must have $\alpha=\beta$.
It follows that $\theta^{-1}\circ\gamma \circ\theta$ is equal to either $\id{\Cliff{1}}\otimes \alpha$ or $\nu\otimes \alpha$ for some $\alpha \in \Aut(A)$.
\end{proof}

Next we will consider the group $\Autgr(\Cliff{2} \otimes A)$ of graded automorphisms for simple trivially graded $C^*$-algebras $A$. By \eqref{eqn:Cliff_even} the even subalgebra of $\Cliff{2} \otimes A$ is isomorphic (as a trivially graded algebra) to
\[
	\Cliff{1} \otimes A \cong A \oplus A\ .
\]
In the following we will identify $\Cliff{2} \otimes A$ with $M_2(A)$ equipped with the grading, in which even elements correspond to diagonal matrices and odd elements to off-diagonal ones. To fix this isomorphism we just need to pick corresponding matrix units in $\Cliff{2}$: Let $\{e_1, e_2\}$ be the standard basis of $\R^2$ and consider the following elements of $\Cliff{2}$
\begin{align*}
	f_{11} = \frac{1 + i\,e_1e_2}{2} \quad , \quad f_{22} = \frac{1 - i\,e_1e_2}{2} \quad , \quad f_{12} = \frac{e_1- i\,e_2}{2}=f_{11}e_1 \quad , \quad f_{21} = \frac{e_1+i\,e_2}{2}=e_1f_{11} \ .
\end{align*}
Note that $f_{11}$ and $f_{22}$ are even, $f_{12}$ and $f_{21}$ are both odd and $f_{ij} \cdot f_{jk} = f_{ik}$ for all $i,j,k \in \{1,2\}$, since $e_1$ and $e_2$ anti-commute.
Hence, we can identify $\Cliff{2}$ with $M_2(\C)$ equipped with its diagonal/off-diagonal grading.

The restriction  of $\alpha\in \Autgr(\Cliff{2} \otimes A)$ to $\Cliff{2}^{\,\rm ev} \otimes A=A\oplus A$ corresponds to the restriction of the $\alpha$ to the diagonal elements. Since $A$ is simple, the primitive ideal space $\text{Prim}(A \oplus A)$ consists of two points. Thus  the restriction of $\alpha$ to $A \oplus A$ is either of the form $(a,b) \mapsto (\alpha_1(a), \alpha_2(b))$ or of the form $(a,b) \mapsto (\alpha_1(b), \alpha_2(a))$.

We identify the multiplier algebras
	\(
		M(\Cliff{2} \otimes A) \cong \Cliff{2} \otimes M(A)
	\).
	For a unitary $u\in U(M(A))$ let $\tilde{u} =
		\left(\begin{smallmatrix}
			1 &  0\\
			0 & u
		\end{smallmatrix}\right) \in  \Cliff{2} \otimes M(A)$.
		\begin{lemma} \label{lem:Autgr_Cl2_short_exact}
	Let $A$ be a simple and trivially graded $C^*$-algebra. The map
	\[\Theta: U(M(A))\times \Z/2 \times \Aut(A) \to \Autgr(\Cliff{2} \otimes A),\quad (u,x,\alpha)\mapsto \alpha = \text{Ad}_{\tilde{u}} \circ (\text{Ad}_{e_1^x} \otimes \alpha_1)\]
		is a homeomorphism. {Here, $U(M(A))$ is equipped with the strict topology.}\end{lemma}

\begin{proof}
	Any automorphism $\alpha \in \Autgr(\Cliff{2} \otimes A)$ extends to a unital graded automorphism of the multiplier algebra
	\(
		M(\Cliff{2} \otimes A) \cong \Cliff{2} \otimes M(A)
	\). We will identify $\Cliff{2}$ with $M_2(\C)$ equipped with its diagonal/off-diagonal grading. The element $\alpha(s) \in \Cliff{2} \otimes M(A)$ for $s = e_1 \otimes 1$ is odd, self-adjoint and satisfies $\alpha(s)^2 = 1$. Hence, it is of the form
	\(
		\alpha(s) =
		\left(\begin{smallmatrix}
			0 &  u^*\\
			u & 0
		\end{smallmatrix}\right) 	
	\)
	for a unique unitary $u \in U(M(A))$.
	Let $\tilde{u} =
		\left(\begin{smallmatrix}
			1 &  0\\
			0 & u
		\end{smallmatrix}\right) $. Then $\tilde{\alpha}:=\text{Ad}_{\tilde{u}^*}\circ \alpha\in \Autgr(\Cliff{2} \otimes A)$  fixes the element $s$.
	
	 Under the isomorphism $\Cliff{2} \cong M_2(\C)$, the element $e_1=f_{12}+f_{21}$ corresponds to $\left(\begin{smallmatrix} 0 & 1 \\ 1 & 0 \end{smallmatrix}\right)$. The restriction of $\tilde{\alpha} \in \Autgr(\Cliff{2} \otimes A)$ to $A \oplus A$ is either of the form $(a,b) \mapsto (\alpha_1(a), \alpha_2(b))$ or of the form $(a,b) \mapsto (\alpha_1(b), \alpha_2(a))$ for some $\alpha_i\in \Aut(A)$. In the first case we have
	\[
		\begin{pmatrix}
			\alpha_1(b) & 0 \\
			0 & \alpha_2(a)	
		\end{pmatrix} = 	
		\tilde{\alpha} \left(
		\begin{pmatrix}
			0 & 1 \\
			1 & 0	
		\end{pmatrix}
		\begin{pmatrix}
			a & 0 \\
			0 & b	
		\end{pmatrix}
		\begin{pmatrix}
			0 & 1 \\
			1 & 0	
		\end{pmatrix}
		\right) =
		\begin{pmatrix}
			0 & 1 \\
			1 & 0
		\end{pmatrix} 		
		\begin{pmatrix}
			\alpha_1(a) & 0 \\
			0 & \alpha_2(b)	
		\end{pmatrix}
		\begin{pmatrix}
			0 & 1 \\
			1 & 0
		\end{pmatrix}		
	\]
	which shows that $\alpha_1= \alpha_2$.	Since
	\[
		\tilde{\alpha} \left(
		\begin{pmatrix}
			0 & c \\
			d & 0	
		\end{pmatrix}
		\right)  = 	
		\tilde{\alpha} \left(
		\begin{pmatrix}
			c & 0 \\
			0 & d	
		\end{pmatrix}
		\begin{pmatrix}
			0 & 1 \\
			1 & 0	
		\end{pmatrix}
		\right) =	
		\begin{pmatrix}
			\alpha_1(c) & 0 \\
			0 & \alpha_1(d)	
		\end{pmatrix}
		\begin{pmatrix}
			0 & 1 \\
			1 & 0
		\end{pmatrix}	
		=
	\begin{pmatrix}
			0 & \alpha_1(a) \\
			\alpha_1(b) & 0	
		\end{pmatrix}				
	 \]
we see that $\tilde{\alpha}=\id{\Cliff{2}}\otimes \alpha_1$.

	 Similar arguments show that in the second case we also have $\alpha_1=\alpha_2$ and that
	 \[\tilde{\alpha}=
	 \text{Ad}_{\left(\begin{smallmatrix}
			0 & 1 \\
			1 & 0
		\end{smallmatrix}\right)} \circ (\id{\Cliff{2}}\otimes \alpha_1) =\text{Ad}_{e_1} \otimes \alpha_1.
	\]
	Since  $\alpha=\text{Ad}_{\tilde{u}}\circ \tilde{\alpha}$,
	we obtain that the map $\Theta$ from the statement
		is a bijection. The continuity properties of $\Theta$ and of its inverse are immediately verified.	\end{proof}

\begin{lemma} \label{lem:Autgr_Cl2_equiv}
Let $A$ be a simple, stable and trivially graded $C^*$-algebra. The group homomorphism
\[
\Z/2 \times \Aut(A) \to \Autgr(\Cliff{2} \otimes A)\quad, \qquad (x,\alpha) \mapsto {\rm Ad}_{e_1^x} \otimes \alpha
\]
is a homotopy equivalence.
\end{lemma}

\begin{proof}
In view of Lemma~\ref{lem:Autgr_Cl2_short_exact}, the result follows from the contractibility of $U(M(A))$
 in the strict topology. 
\end{proof}

\begin{lemma} \label{lem:homotopy_stable_alg}
	Let $A$ be a stable $C^*$-algebra and let $n \in \N$. The group homomorphisms
	\begin{align*}
		\Aut(A) \to \Aut(\K \otimes A) \quad &, \quad \alpha \mapsto \id{\K} \otimes \alpha \ ,\\
		\Aut(A) \to \Aut(M_n(\C) \otimes A) \quad &,\quad \alpha \mapsto \id{M_n(\C)} \otimes \alpha\ .		
	\end{align*}
	are homotopy equivalences.
\end{lemma}

\begin{proof}
	Let $M_n = M_n(\C)$. Using an isomorphism $A \cong \K \otimes A$ we see that it suffices to show that
	\begin{align*}
		\Aut(\K \otimes A) \to \Aut(\K \otimes \K \otimes A) \quad &, \quad \alpha \mapsto \id{\K} \otimes \alpha \ ,\\
		\Aut(\K \otimes A) \to \Aut(M_n \otimes \K \otimes A) \quad &,\quad \alpha \mapsto \id{M_n} \otimes \alpha
	\end{align*}
	are homotopy equivalences. In fact, if the first map induces a homotopy equivalence, then so does the second. To see this note that the composition
	\[
		\begin{tikzcd}[column sep=2.1cm]
			\Aut(\K \otimes A) \ar[r, "\alpha \mapsto \id{M_n} \otimes \alpha" ] & \Aut(M_n \otimes \K \otimes A) \ar[r,"\beta \mapsto \id{\K} \otimes \beta"] & \Aut(\K \otimes M_n \otimes \K \otimes A)
		\end{tikzcd}
	\]
	maps $\alpha$ to $\id{\K \otimes M_n} \otimes \alpha$. If $\alpha \mapsto \id{\K} \otimes \alpha$ is a homotopy equivalence, then so is $\alpha \mapsto \id{\K \otimes M_n} \otimes \alpha$ (because $\K \otimes M_n \cong \K$) and therefore also $\alpha \mapsto \id{M_n} \otimes \alpha$.
	
	Let $e \in \K$ be a minimal projection. 	By \cite[Lemma~2.4]{DP1} there exists a continuous map $\psi \colon [0,1] \to \text{Hom}(\K,\K \otimes \K)$ with the property that $\psi(0)(T) = e \otimes T$ for all $T \in \K$ and $\psi(t)$ is an isomorphism for all $t \in (0,1]$ (see also the proof of \cite[Thm.~2.5]{DP1}). The homotopy inverse of $\alpha \mapsto \id{\K} \otimes \alpha$ is given by $\beta \mapsto (\psi_1^{-1} \otimes \id{A}) \circ \beta \circ (\psi_1 \otimes \id{A})$. The homotopy
	\[
		H \colon \Aut(\K \otimes A) \times I \to \Aut(\K \otimes A) \  , \  (\alpha,t) \mapsto
		\begin{cases}
			\alpha & \text{if } t = 0, \\
			(\psi_t^{-1} \otimes \id{A}) \circ (\id{\K} \otimes \alpha) \circ (\psi_t \otimes \id{A}) & \text{else.}
		\end{cases}
	\]
	shows that $(\psi_1^{-1} \otimes \id{A}) \circ (\id{\K} \otimes \alpha) \circ (\psi_1 \otimes \id{A})$ is homotopic to the identity map. The proof that the other composition of the two maps is also homotopic to the identity can be reduced to this argument and is therefore omitted.
\end{proof}

\begin{proposition} \label{lem:Cliff_hom_eq}
	Let $A$ be a simple, stable and trivially graded $C^*$-algebra. For every $n \in \N$ the group homomorphism
	\[
		\Psi_n \colon \Autgr(\Cliff{n} \otimes A) \to \Autgr(\Cliff{n+1} \otimes A)\ ,
	\]
	induced by $\Autgr(\Cliff{n}\otimes A) \to \Autgr(\Cliff{1} \otimes \Cliff{n}\otimes A)$, $\alpha \mapsto \id{} \otimes \alpha$ and the isomorphism $\Cliff{1} \otimes \Cliff{n} \cong \Cliff{n+1}$, is a homotopy equivalence.
\end{proposition}

\begin{proof}
	By the periodicity of Clifford algebras it suffices to show the statement for $n = 1$ and $n = 2$. In the first case consider the diagram
	\[
		\begin{tikzcd}
			\Autgr(\Cliff{1} \otimes A) \ar[r,"\Psi_1"] & \Autgr(\Cliff{2} \otimes A) \\
		\Z/2 \times	\Aut(A) \ar[u,"\cong"] \ar[ur] &
		\end{tikzcd}
	\]
	where the vertical arrow is the isomorphism from Lemma~\ref{lem:Autgr_Cl1_short_exact} and the diagonal arrow is given by $(x,\alpha) \mapsto \text{Ad}_{e_1^x} \otimes \alpha$, the homomorphism from Lemma~\ref{lem:Autgr_Cl2_equiv}. To see that this diagram commutes, we observe that if $\nu$ is the grading automorphism of $\Cliff{1}$, then $\id{\Cliff{1}}\otimes \nu \in \Autgr(\Cliff{1}\otimes \Cliff{1})\cong \Autgr(\Cliff{2})$  is implemented in $\Cliff{2}$ by conjugation by $e_1$. So this case follows from Lemma~\ref{lem:Autgr_Cl2_equiv}.
	
To handle the case $n =2$ it suffices to show that the homomorphism \[\Psi_2\circ \Psi_1\colon \Aut(\Cliff{1})\times \Aut(A)\cong  \Autgr(\Cliff{1}\otimes A) \to \Autgr(\Cliff{1} \otimes\Cliff{1} \otimes \Cliff{1} \otimes A),\]
$(\beta,\alpha) \to
\id{\Cliff{1}} \otimes\id{\Cliff{1}}\otimes \beta\otimes \alpha$ is a homotopy equivalence.
Let $\sigma$ be the automorphism of $\Cliff{1} \otimes\Cliff{1} \otimes \Cliff{1} \otimes A$ that permutes cyclically the first three tensor factors and acts as identity on $A$. Since conjugation by $\sigma$ is a homeomorphism of $ \Autgr(\Cliff{1} \otimes\Cliff{1} \otimes \Cliff{1} \otimes A)$ it suffices to show that the map
$\psi(\beta,\alpha)=\beta\otimes \id{\Cliff{1}} \otimes\id{\Cliff{1}}\otimes \alpha$ is a homotopy equivalence.
The grading automorphism  of $B:=\Cliff{1} \otimes \Cliff{1}\cong M_2(\C)$ is inner,  implemented by the unitary $v=\left(\begin{smallmatrix}
			1 &  0\\
			0 & -1
		\end{smallmatrix}\right) $.
 It follows that the graded tensor product $\Cliff{1}\otimes B$ is isomorphic to $\Cliff{1}\otimes \overline{B}$ where $\overline{B}$ stands for the $C^*$-algebra $B$ with the trivial grading. The isomorphism $\theta:\Cliff{1}\otimes B \to \Cliff{1}\otimes \overline{B}$ is given by $  x\otimes b \mapsto x\otimes vb$.
  Next we observe that the homomorphism $\psi'(-):=(\theta\otimes \id{A})\circ\psi(-)\circ(\theta\otimes \id{A})^{-1} \colon \Aut(\Cliff{1})\times \Aut(A) \to \Autgr(\Cliff{1} \otimes\overline{\Cliff{1} \otimes \Cliff{1}} \otimes A)$
 maps $(\beta,\alpha)$ to $\beta \otimes \id{M_2(\C)} \otimes \alpha$.
 Indeed if $x \in \Cliff{1}, b \in \Cliff{1}\otimes \Cliff{1}$ and $a\in A$,
 \begin{equation*}
\begin{split}
\psi'(\beta,\alpha)(x\otimes b\otimes a)& =(\theta\otimes \id{A})(\psi(\beta,\alpha)(x\otimes v^*b\otimes a)) \\
& =(\theta\otimes \id{A})(\beta(x)\otimes v^*b\otimes \alpha(a))=\beta(x)\otimes b\otimes \alpha(a)
\end{split}
\end{equation*}
 It follows that $\psi'$ is a homotopy equivalence by Lemma~\ref{lem:Autgr_Cl2_equiv} applied for $M_2(\C)\otimes A$ and Lemma~\ref{lem:homotopy_stable_alg}. We conclude that $\psi$ and hence $\Psi_2\circ \Psi_1$ is a homotopy equivalence.	
\end{proof}

\section{Spectra in algebraic topology}

\subsection{Spectra, commutative $\mI$-monoids and units}
\subsubsection{Spectra and cohomology theories}\label{subsub-scth}
To understand the link between cohomology theories and infinite loop spaces, we need to recall some basic facts from stable homotopy theory. A sequence $(E_n)_{n \in \N_0}$ of pointed topological spaces together with weak equivalences $E_n \to \Omega E_{n+1}$ is called an $\Omega$-spectrum. A topological space $Z$ is an infinite loop space, if there is an $\Omega$-spectrum $(E_n)_{n \in \N_0}$ with $Z \simeq E_0$. This is often written as $Z = \Omega^\infty E$. It is an easy exercise in homotopy theory to check that
\[
	\widetilde{h}^n(X) = [X,E_n]_+\ ,
\]
where the right hand side denotes based homotopy classes of based continuous maps, defines a reduced cohomology theory on the category of pointed finite CW-complexes. For all pairs $(X,A)$ of a finite CW-complex $X$ and a subcomplex $A$, the groups $h^n(X,A) = \widetilde{h}^n(X/A)$ provide an (unreduced) cohomology theory (where we set $X/\emptyset := X_+$, ie.\ $X$ with an added disjoint base point). Using the mapping cone instead of the quotient this is easily extended to pairs $(X,A)$ of a finite CW-complex with a subspace $A$. Note in particular that
\[
	h^n(X) = h^n(X,\emptyset) = [X_+, E_n]_+ = [X, E_n]\ .
\]
\subsubsection{Commutative symmetric ring spectra}
The terminology ``space'' will always refer to compactly generated Hausdorff spaces. A symmetric spectrum  consists of a sequence of pointed spaces  $E_n$ for $n \geq 0$, a basepoint preserving left action of the symmetric group $\Sigma_n$ on $E_n$ and $\Sigma_n$-equivariant based maps $\sigma_n \colon E_n \wedge S^1 \to E_{n+1}$ for $n \geq 0$. There is a category $\mathcal{S}p^{\Sigma}$ of symmetric spectra, which has a symmetric monoidal structure (via the smash product). The (commutative) monoid objects in $\mathcal{S}p^{\Sigma}$ are called (commutative) ring spectra. For details about the category of symmetric spectra we refer the reader to \cite{paper:HoveyShipleySmith}. An ``unpacked'' version of the definition of a commutative symmetric ring spectrum can be found in \cite[Def.~2.1]{DP2}.

\subsubsection{Units}
Let $R$ be a commutative unital ring. The set of invertible elements in $R$ forms the group of units denoted by $GL_1(R)$. Likewise, for a commutative ring spectrum $E$ and a space $X$, $E^0(X)$ is a commutative ring with units $GL_1(E^0(X))$. By May, Quinn, Ray, Tornhave \cite{paper:May}, if $E$ is an $E_\infty$-ring spectrum, one can lift this construction to a spectrum of units denoted $gl_1(E)$ so that $gl_1(E)^0(X)=GL_1(E^0(X))$. In particular this is the case if $E$ is a commutative symmetric ring spectrum in the sense of \cite{paper:HoveyShipleySmith}.

\subsubsection{Commutative $\mI$-monoids}
 In \cite{paper:Schlichtkrull} Schlichtkrull gave a description of the infinite loop space underlying $gl_1(E)$ in terms of commutative $\mI$-monoids, which will be convenient for us and which we outline here.

We recall the following definition from \cite{paper:Schlichtkrull, paper:SagaveSchlichtkrull}. Let $\mI$ be the category with objects given by the finite sets $\mathbf{n} = \{1,\dots,n\}$ including the empty set $\mathbf{0}$ and morphisms injective maps between these sets. An \emph{$\mI$-space} is a (covariant) functor from $\mI$ to the category of topological spaces and continuous maps. The category $\mI$ is symmetric monoidal, where $\mathbf{n} \sqcup \mathbf{m} = \{1, \dots, n+m\}$ on objects, $\mathbf{0}$ is the tensor identity, and tensor products of morphisms act by identifying the subset $\{1,\dots n\} \subset \{1, \dots, n+m\}$ with $\mathbf{n}$ and $\{n+1, \dots, n+m\}$ with $\mathbf{m}$. The category of $\mI$-spaces inherits a symmetric monoidal structure and monoids with respect to this structure are called $\mI$-monoids. To spell out explicitly what this means note that an $\mI$-space $X$ is an $\mI$-monoid if it comes equipped with a multiplication
\[
	\mu_{n,m} \colon X(\mathbf{n}) \times X(\mathbf{m}) \to X(\mathbf{n} \sqcup \mathbf{m})
\]
that is a morphism between $\mI \times \mI$-spaces, such that the obvious associativity and unitality diagrams commute. An $\mI$-monoid $X$ is \textit{commutative} if the following diagram commutes
\begin{equation} \label{eqn:comm_I_mon}
	\begin{tikzcd}[row sep=0.5cm]
		X(\mathbf{n}) \times X(\mathbf{m}) \ar[r,"\mu_{n,m}"] \ar[d] & X(\mathbf{n} \sqcup \mathbf{m}) \ar[d , "(\tau_{m,n})_*"] \\
		X(\mathbf{m}) \times X(\mathbf{n}) \ar[r,"\mu_{m,n}" below] & X(\mathbf{m} \sqcup \mathbf{n})
	\end{tikzcd}	
\end{equation}
where the left vertical arrow interchanges the factors and the right vertical arrow is induced by the block permutation $\tau_{m,n}:\mathbf{n} \sqcup \mathbf{m} \to \mathbf{m} \sqcup \mathbf{n}$.

\subsubsection{The unit spectrum of a commutative symmetric ring spectrum}\label{subsec:csrs}
Given a commutative symmetric ring spectrum $E$ and $n \in \N_0$, let $\Omega^n(E_n)^*$ be the union of path components in $\Omega^n(E_n)$ that have a stable multiplicative homotopy inverse in the sense that for each $f \colon S^n \to E_n$ in $\Omega^n(E_n)^*$ there exists $g \in \Omega^m(E_m)$ such that
\[
\begin{tikzcd}[column sep=1.2cm]
	S^n \wedge S^m \ar[r,"f \wedge g"] & E_n \wedge E_m \ar[r,"\mu_{n,m}"] & E_{n+m}
\end{tikzcd}
\]
is homotopic to the unit map $S^{n+m} \to E_{n+m}$ of the ring spectrum in $\Omega^{n+m}(E_{n+m})$ and similarly for $\mu_{m,n} \circ (g \wedge f)$. This can be extended to a functor $\mathbf{n} \mapsto \Omega^n(E_n)^*$ as follows: Permutations act on $\Omega^n(E_n)^*$ by permuting the circle coordinates and using the $\Sigma_n$-action on $E_n$. Moreover, the map $\mathbf{n} \to \mathbf{n+1}$ given by $i \mapsto i+1$ maps $f \in \Omega_n(E_n)^*$ to the composition
\[
\begin{tikzcd}[column sep=1.2cm]
	S^1 \wedge S^n \ar[r,"u \wedge f"] & E_1 \wedge E_n \ar[r,"\mu_{1,n}"] & E_{1+n}
\end{tikzcd}
\]
where $u \colon S^1 \to E_1$ denotes the unit map of the ring spectrum $E$. Now define
\begin{equation}
	GL_1(E) = \hocolim_{\mI} \Omega^n(E_n)^*	\ .
\end{equation}
Under suitable connectivity assumptions (see \cite[Sec.~2.3]{paper:Schlichtkrull}) we have $\pi_0(GL_1(E)) \cong GL_1(\pi_0(E))$ and in this case the $\mI$-monoid $\mathbf{n} \mapsto \Omega^n(E_n)^*$ has the homotopy type of the unit space of the ring spectrum $E$. Since we assumed $E$ to be commutative, there is also a connective spectrum $gl_1(E)$ associated to this $\mI$-monoid, which is called the \emph{spectrum of units} of $E$. For its underlying infinite loop space we have $\Omega^\infty(gl_1(E)) \simeq GL_1(E)$.

\subsection{Spectra representing $K$-theory and its units} \label{sec:spectra_and_units}
\subsubsection{The K-theory spectrum $KU^D$}\label{subsec:kud}
To any strongly self-absorbing $C^*$-algebra $D$ we can associate a commutative symmetric ring spectrum $KU^D$ defined by the sequence of spaces
\[
	KU_n^D = \hom_{\text{gr}}(\Sgr, (\Cliff{1} \otimes D \otimes \K)^{\otimes n})\ ,
\]
where the set of $*$-homomorphisms is equipped with the point-norm topology and the algebra $D \otimes \K$ is trivially graded. The $*$-homomor\-phisms in \eqref{eqn:tensor_flip} give a $\Sigma_n$-action on $(\Cliff{1} \otimes D \otimes \K)^{\otimes n}$ by permuting the tensor factors. The $\Sigma_n$-action on $KU_n^D$ is given by composition with these maps. The multiplication of $KU^D$ stems from the maps
\[
	KU_n^D \times KU_m^D \to KU_{n+m}^D \quad, \quad (\varphi, \psi) \mapsto (\varphi \otimes \psi) \circ \Delta =: \varphi \ast \psi
\]
and uses the comultiplication $\Delta \colon \Sgr \to \Sgr \otimes \Sgr$ described in the previous section. Note that we can identify $\Omega KU^D_{1}$ with the space
\[
	\hom_{\text{gr}}(\Sgr, C_0(\R,\Cliff{1} \otimes D \otimes \K))\ ,
\]
where $C_0(\R)$ is trivially graded. Let $e \in \K$ be a fixed minimal projection. The map
\[
	\R \to \Cliff{1} \otimes D \otimes \K \quad , \quad t \mapsto t\,(e_1 \otimes 1_D \otimes e)
\]
is an odd, self-adjoint, regular unbounded multiplier on $C_0(\R, \Cliff{1}\otimes D \otimes \K)$ and therefore corresponds to a $*$-homomorphism $\eta_1 \in \Omega KU_1^D$. The structure map $KU_n^D \wedge S^1 \to KU_{n+1}^D$ is the adjoint of
\[
	KU_n^D \to \Omega KU_{n+1}^D \quad , \quad \varphi \mapsto \varphi \ast \eta_1
\]
We refer the reader to \cite[Thm.~4.2]{DP2} for a proof that $KU^D$ indeed defines a commutative symmetric ring spectrum. Inductively, we define $\eta_l = \eta_{l-1} \ast \eta_1 \in \Omega^lKU^D_l$.

\subsubsection{The unit spectrum of $KU^D$}
In the case of the commutative symmetric ring spectrum $KU^D$ introduced above, the map $\mathbf{n} \to \mathbf{n+1}$ defined by $i \mapsto i+1$ for $n \geq 1$ induces the group homomorphism
\[
	K_i(C_0(\R^n,\Cliff{n} \otimes D^{\otimes n})) \to K_i(C_0(\R^{n+1},\Cliff{n+1} \otimes D^{\otimes (n+1)}))
\]
on $\pi_i(\Omega^n(KU^D)^* \cong K_i(C_0(\R^n,\Cliff{n} \otimes D^{\otimes n}))$ given by the tensor product with the element $b \otimes [1_D] \in K_0(C_0(\R,\Cliff{1} \otimes D))$, where $b \in K_0(C_0(\R,\Cliff{1}))$ is the Bott class \cite[19.9.3]{Bla:k-theory}. Since this is an isomorphism, the $\mI$-monoid $\mathbf{n} \mapsto \Omega^n(KU^D_n)^*$ is convergent, in fact, all maps $\mathbf{n} \to \mathbf{m}$ for $n \geq 1$ induce weak homotopy equivalences.

\begin{definition}
	The \emph{unit spectrum of $KU^D$} is the connective spectrum $gl_1(KU^D)$ associated to the commutative $\mI$-monoid $\mathbf{n} \mapsto \Omega^n(KU^D_n)^*$ defined in Sec.~\ref{subsec:csrs} and \ref{subsec:kud}. We define the \emph{space of units} as the underlying infinite loop space, ie.\ $GL_1(KU^D) = \Omega^{\infty}(gl_1(KU^D))$.
\end{definition}

\subsection{Localization of spectra}
Given a commutative ring $R$ with unit and a subset $S \subset R$ we can invert the elements in $S$ and form the localization $R[S^{-1}]$ at the multiplicative closure of~$S$. This procedure extends to spectra. In fact,  Bousfield gave a more general construction in \cite{Bousfield} that localizes a spectrum with respect to a homology theory. We will only need the simplest case here, which is the spectral analogue of $\Z_P$ for a set of primes~$P$, ie.\ we would like to define a commutative symmetric ring spectrum $E[P^{-1}]$ starting from a commutative symmetric ring spectrum $E$. As explained for example in \cite[Sec.~3.3]{Casacuberta} such a localization is \emph{smashing}, ie.\ $E[P^{-1}] \simeq E \wedge \spS[P^{-1}]$, where $\spS[P^{-1}]$ is the corresponding localization of the sphere spectrum $\spS$. The spectrum $\spS[P^{-1}]$ can be obtained as a commutative symmetric ring spectrum as follows: Following \cite[Chap.~V, Def.~1.12]{EKMM} and \cite[Chap.~V, Prop.~2.3]{EKMM} there is a model as a commutative $S$-algebra. The functor $\Phi$ constructed in \cite{Schwede_Smod_symspec} produces a commutative symmetric ring spectrum~$\spS[P^{-1}]$ from this, which we will also assume to be cofibrant in the positive model structure on commutative symmetric ring spectra. This ensures that the smash product in Def.~\ref{def:Plocalization} has the correct homotopy type. Let $\mathbb{P} \subset \N$ be the set of all prime numbers.

\begin{definition} \label{def:Plocalization}
	Let $P \subset \mathbb{P}$. Let $E$ be a symmetric spectrum and let $\spS[P^{-1}]$ be the cofibrant commutative symmetric ring spectrum obtained from the sphere spectrum $\spS$ as described above. We define the $P$-localization of $E$ to be
	\[
		E[P^{-1}] = E \wedge \spS[P^{-1}]\ .
	\]
\end{definition}

Let $p \in \mathbb{P}$ and $P \subset \mathbb{P}$. From the above discussion we obtain the following commutative symmetric ring spectra
\[
	KU := KU^\C \quad , \quad KU_p := KU[\{p\}^{-1}] \quad , \quad KU_{(p)} := KU[(\mathbb{P} \setminus \{p\})^{-1}] \quad , \quad KU_P:=KU[P^{-1}] \ .
\]

\begin{lemma} \label{lem:Plocalization}
Let $P \subset \mathbb{P}$ and let $M_P$ be the associated UHF-algebra. There is a zig-zag of stable equivalences of commutative symmetric ring spectra $KU[P^{-1}] \simeq KU^{M_P}$ under $KU$.
\end{lemma}

\begin{proof}
Using \cite[Thm.~4.2]{DP2} we see that the homotopy groups of the localization satisfy
\[
	\pi_i(KU^D[P^{-1}]) \cong \pi_i(KU^D) \otimes \Z[P^{-1}] \cong K_i(D) \otimes \Z[P^{-1}]
\]
for $i \in \Z$. Since $K_{\text{even}}(M_P) \cong \Z[P^{-1}]$, $K_{\text{odd}}(M_P) = 0$ and $\Z[P^{-1}] \otimes \Z[P^{-1}] \cong \Z[P^{-1}]$, the map on the right and the bottom map in the following commutative square
\[
\begin{tikzcd}[column sep=0.7cm, row sep=0.7cm]
	KU \ar[r] \ar[d] \ar[dr] & KU^{M_P} \ar[d,"\simeq"] \\
	KU[P^{-1}] \ar[r,"\simeq" below] & KU^{M_P}[P^{-1}]
\end{tikzcd}
\]	
are stable equivalences of commutative symmetric ring spectra.
\end{proof}

\begin{theorem}[cf. \cite{DP2}]\label{thm:B} {There are stable equivalences of spectra:}
\begin{align*}
	gl_1(KU^\C) &\simeq gl_1(KU^\ZZ)\simeq gl_1(KU^{\OO_\infty}) \\
	gl_1(KU^{M_{P}}) &\simeq gl_1(KU^{M_{P}\otimes \OO_\infty })
\end{align*}
\end{theorem}
\begin{proof} Since the unital $*$-homomorphisms $\C \to \ZZ \to \OO_\infty$ are $KK$-equivalences, the arguments from the proof of \cite[Thm.4.7]{DP2} show that they induce $\pi_*$-isomorphisms in the sense of \cite[Def.~5.3.10]{book:BarnesRoitzheim}. By \cite[Prop.~5.3.12]{book:BarnesRoitzheim} they give stable equivalences of commutative symmetric ring spectra $KU^D \simeq KU^{D \otimes \ZZ} \simeq KU^{D\otimes \OO_\infty}$ for any strongly self-absorbing $C^*$-algebra $D$. Since $KU^D$ is a positive $\Omega$-spectrum, the above $*$-homomorphisms induce weak equivalences $\Omega^n KU^D_n \simeq \Omega^n KU^{D \otimes \ZZ}_n \simeq \Omega^n KU^{D\otimes \OO_\infty}_n$. Moreover, the invertible components of $\Omega^n KU^D_n$ correspond to those elements in $\pi_0(\Omega^n KU^D_n) \cong \pi_0(\Omega KU^D_1)$ that are invertible with respect to the multiplication
\[
	\pi_0(\Omega KU^D_1) \times \pi_0(\Omega KU^D_1) \to \pi_0(\Omega^2 KU^D_2) \cong \pi_0(\Omega KU^D_1)\ ,
\]
Since these are identified by the maps induced by the $*$-homomorphisms, they induce morphisms of commutative $\mI$-monoids $(\Omega^n KU^D_n)^* \to (\Omega^n KU^{D \otimes \ZZ}_n)^*$ and $(\Omega^n KU^D_n)^* \to (\Omega^n KU^{D \otimes \OO_\infty}_n)^*$\ that are weak equivalences on the respective homotopy colimits.
Hence, they give rise to stable equivalences of the corresponding unit spectra:
\[
	gl_1(KU^D)\simeq gl_1(KU^{D \otimes \ZZ})\simeq gl_1(KU^{D\otimes \OO_\infty})\ .
\]
Furthermore, by Lem.~\ref{lem:Plocalization} there are zig-zags of stable equivalences of commutative symmetric ring spectra $KU_{P} \simeq KU^{M_P}$ which induce a stable equivalence
\[
	gl_1(KU^{M_{P}}) \simeq gl_1(KU_{P})\ . \qedhere
\]

\end{proof}
\section{Generalized Dixmier-Douady theory}\label{DD}
\subsection{Bundles of strongly selfabsorbing $C^*$-algebras}
If $D$ is a strongly self-absorbing $C^*$-algebra and $\K$ is the $C^*$-algebra of compact operators on a separable infinite dimensional Hilbert space then  $(D\otimes \K)\otimes (D\otimes \K)\cong D\otimes \K$. In particular, the fibrewise tensor product induces a semigroup structure on the set of locally trivial $C^*$-algebra bundles with fibre $D \otimes \K$. We have shown in \cite{DP1,DP2} that the isomorphism classes of such bundles over a compact metrizable space $X$ in fact form an abelian group under the operation of fibrewise tensor product. This was achieved by showing that the group $\Aut(D \otimes \K)$ admits an infinite loop space structure such that the tensor product operation on the classifying space $B\Aut(D \otimes \K)$ coincides with the composition of loops in $\Omega B(B\Aut(D \otimes \K))$ up to homotopy. Since infinite loop spaces give rise to connective spectra, there is a cohomology theory $E_D^*(X)$ such that its zeroth group $E_D^0(X)$ computes the homotopy classes $E_D^0(X)\cong [X,\Aut(D\otimes \K)]$ and its first group $E_D^1(X)=[X,B\Aut(D \otimes \K)]$ computes the group of isomorphism classes of locally trivial bundles with fibers $D \otimes \mathcal{K}$ endowed with the operation of tensor product. The coefficients of the cohomology theory $E_D^*(X)$ are given by the homotopy groups of $\Aut(D\otimes \K)$ computed in \cite{DP1}:
\begin{equation}\label{homotopy}
E^{-i}(*)\cong \pi_i(\Aut(D\otimes \K))=
\begin{cases}
K_0(D)^{\times}_{+},\,\,\text{if}\,\, i=0\\
K_i(D),\,\,\text{if}\,\, i\geq 1.\\
\end{cases}
\end{equation}
The coefficients determine a cohomology theory only rationally, ie.\ up to torsion. In pursuit of a better understanding of $E_D^*(X)$ it was further shown in \cite{DP2} that its spectrum is closely related to $gl_1(KU^D)$ and in fact is equivalent to it, if $D$ is purely infinite.
In the following we will give more context in order to explain the results of \cite{DP2} in preparation for their extension to  graded algebras.

Let $KU^D$ be the commutative symmetric ring spectrum introduced in Sec.~\ref{sec:spectra_and_units}. The spaces representing its units~$gl_1(KU^D)$ are denoted by $GL_1(KU^D)$, $BGL_1(KU^D)$, $BBGL_1(KU^D)$, etc. Consider the commutative $\mI$-monoid $G_D$ defined by
\[
	G_D(\mathbf{n}) = \Aut((D \otimes \K)^{\otimes n})\ .
\]
The map $G_D(\mathbf{m} \to \mathbf{n})$ associated to a morphism $\mathbf{m} \to \mathbf{n}$ permutes the tensor factors labelled by the elements in the image of the morphism and acts as $\id{D \otimes \K}$ on the remaining ones. The monoid structure is induced by the tensor product of automorphisms (see \cite[Sec.~4.2]{DP2}). Note that $G_D$ takes values in topological groups. Denote the $\mI$-monoid multiplication by $\mu_{m,n}$ and the composition of automorphisms by
\[
	\nu_n \colon G_D(\mathbf{n}) \times G_D(\mathbf{n}) \to G_D(\mathbf{n})\ .
\]
Observe that $(\alpha_1 \otimes \alpha_2) \circ (\beta_1 \otimes \beta_2) = (\alpha_1 \circ \beta_1) \otimes (\alpha_2 \circ \beta_2)$ implies that the following diagram commutes:
\begin{equation} \label{eqn:EH-I-group}
	\begin{tikzcd}[column sep=2.5cm]
		G_D(\mathbf{m}) \times G_D(\mathbf{m}) \times G_D(\mathbf{n}) \times G_D(\mathbf{n}) \ar[r,"(\mu_{m,n} \times \mu_{m,n}) \circ \tau"] \ar[d,"\nu_m \times \nu_n" left] & G_D(\mathbf{m} \sqcup \mathbf{n}) \times G_D(\mathbf{m} \sqcup \mathbf{n}) \ar[d,"\nu_{m+n}"] \\
		G_D(\mathbf{m}) \times G_D(\mathbf{n}) \ar[r,"\mu_{m,n}" below] & G_D(\mathbf{m} \sqcup \mathbf{n})
	\end{tikzcd}
\end{equation}
where $\tau$ flips the two middle factors of the product in the upper left hand corner. Therefore $G_D$ is an EH-$\mI$-group (Eckmann-Hilton $\mI$-group) in the sense of \cite[Def.~3.1]{DP2}. This EH-$\mI$-group acts on $KU^D$, and the action gives rise to a morphism of commutative $\mI$-monoids \cite[Thm.~3.8]{DP2}
\[
	G_D \to \Omega^\infty(KU^D)^*\ ,
\]
which induces an isomorphism on all homotopy groups $\pi_n$ for $n > 0$ and produces an equivalence of the underlying infinite loop spaces
\(
	\left(G_D\right)_{h\mI} \to GL_1(KU^D)
\)
if $D$ is purely infinite \cite[Thm.~4.6]{DP2}. Since all maps $\mathbf{m} \to \mathbf{n}$ with $m \geq 1$ induce homotopy equivalences, we have $\left(G_D\right)_{h\mI} \simeq \Aut(D \otimes \K)$ for the homotopy colimit \cite[Lem.~3.5]{DP2}. The first delooping $B_{\mu}(G_D)_{h\mI}$ of $(G_D)_{h\mI}$ with respect to the $\mI$-monoid structure has the homotopy type of $B\Aut(D \otimes \K)$ according to \cite[Thm.~3.6]{DP2}. It was also shown in \cite[Thm.~4.9]{DP2} that the infinite loop space structure on $B\Aut(D \otimes \K)$ obtained in this way agrees with the one found in \cite{DP1}. Hence the connective spectrum underlying $G_D$ indeed represents  the cohomology theory $X \mapsto E_D^*(X)$.

\begin{theorem}[\cite{DP2},Thm.4.6] \label{thm:A}
	Let $D$ be a strongly self-absorbing $C^*$-algebra. The action of $\G_D$  on the ring spectrum $KU^D$ induces a map
	\begin{equation} \label{eqn:1map_of_Gamma}
\Gamma(\G_D) \to \Gamma(\Omega^\infty(KU^D)^*)
	\end{equation}
	of the associated $\Gamma$-spaces. In turn, this induces an isomorphism on all homotopy groups $\pi_n$ with $n > 0$ of the corresponding connective spectra and the homomorphism
	\[
		\{\pm 1\} \times K_0(D)_{+}^\times \to K_0(D)^\times \quad , \quad (a,b) \mapsto a \cdot b
	\]
	on $\pi_0$. In particular, \eqref{eqn:1map_of_Gamma} induces an equivalence in the stable homotopy  category of spectra  if $D$ is purely infinite and satisfies the UCT and hence  $E^*_D(X)\cong gl_1(KU^D)^*(X).$
\end{theorem}

As stated in \eqref{homotopy} we have $\pi_0(\Aut(D \otimes \K)) \cong K_0(D)_+^\times$, whereas $\pi_0(GL_1(KU^D)) \cong K_0(D)^\times$. Hence, the assumption that $D$ is purely infinite is needed, because the tensor product with $\OO_\infty$ trivialises the order structure on $K_0(D)$. The group $gl_1(KU)^1(X)$ is closely related to the graded Brauer group discussed for example in \cite{paper:Parker-Brauer}. This suggests a different approach to correct $\pi_0$ making use of graded $C^*$-algebras, in particular the Clifford algebras that appear in the definition of the spectrum $KU^D$. Following the proof outlined above we will see in the next sections that this idea indeed motivates another very natural EH-$\mI$-group representing $gl_1(KU^D)$.

Apart from this we obtain from Theorem \ref{thm:A} and Theorem \ref{thm:B} the isomorphisms:
\begin{align*}
\label{eqn:eqlabel}
 E^*_{\OO_\infty}(X) \cong gl_1(KU)^*(X) \qquad \text{and} \qquad E^*_{M_{P}\otimes \OO_\infty}(X) \cong gl_1(KU_{P})^*(X)
\end{align*}
These groups  will be further discussed in the sequel. 
	
\subsection{Bundles of strongly selfabsorbing $C^*$-algebras with a Clifford grading}

\subsubsection{Eckmann-Hilton $\mathcal{I}$-groups and graded $C^*$-algebras}
Let $D$ be a strongly self-absorbing $C^*$-algebra. Motivated by the above considerations we define the $\mI$-space $\Ggr_D$ by
\[
	\Ggr_D(\mathbf{n}) = \Autgr( (\Cliff{1} \otimes D \otimes \K)^{\otimes n})
\]
where the graded automorphism groups are equipped with the point-norm topology. Note that $\Ggr_D(\mathbf{0}) = \Autgr( (\Cliff{1} \otimes D \otimes \K)^{\otimes 0}) = \Aut(\C)$ is the one-point space. The value of $\Ggr_D$ on the morphisms of $\mI$ is fixed by the following: If $\sigma \colon \mathbf{n} \to \mathbf{n}$ denotes a permutation, then $\Ggr_D(\sigma)$ is defined by $\alpha \mapsto \epsilon_\sigma \circ \alpha \circ \epsilon_{\sigma}^{-1}$, where
\begin{equation}\label{eqn:permutation}
\epsilon_{\sigma} \colon (\Cliff{1} \otimes D \otimes \K)^{\otimes n} \to (\Cliff{1} \otimes D \otimes \K)^{\otimes n}
\end{equation}
is the graded permutation of the tensor factors corresponding to $\sigma$ (see \eqref{eqn:tensor_flip}). In fact, because each $\alpha \in \Ggr_D(\mathbf{n})$ preserves the degree of homogeneous elements, the sign that appears in \eqref{eqn:tensor_flip} cancels out when we conjugate by $\epsilon_{\sigma}$. Hence, the conjugation actually agrees with the one by the ungraded tensor flip (which is, however, not an algebra automorphism).

If $m = n+1$ and $\iota \colon \mathbf{n} \to \mathbf{m}$ denotes the map $i \mapsto i+1$, then $\Ggr_D(\iota) \colon \Ggr_D(\mathbf{n}) \to \Ggr_D(\mathbf{m})$ maps $\alpha$ to $\id{} \otimes \alpha$. Just as in the case of $G_D$ (see \cite[Sec.~4.2]{DP2}) defining $\Ggr_D(\iota)$ and $\Ggr_D(\sigma)$ fixes the $\mI$-space structure completely. We can equip $\Ggr_D$ with a multiplication that turns it into an $\mI$-monoid as follows:
\[
	\mu_{m,n} \colon \Ggr_D(\mathbf{m}) \times \Ggr_D(\mathbf{n}) \to \Ggr_D(\mathbf{m} \sqcup \mathbf{n}) \qquad ; \qquad (\alpha_1, \alpha_2) \mapsto \alpha_1 \otimes \alpha_2\ .
\]
In particular, the single point in $\Ggr_D(\mathbf{0})$ acts as a unit for $\mu_{m,n}$. Consider the block permutation $\sigma \colon \mathbf{m} \sqcup \mathbf{n} \to \mathbf{n} \sqcup \mathbf{m}$. Since $\epsilon_{\sigma} \circ (\alpha_1 \otimes \alpha_2) \circ \epsilon_{\sigma}^{-1} = \alpha_2 \otimes \alpha_1$, the $\mI$-monoid $\Ggr_D$ is commutative in the sense of \eqref{eqn:comm_I_mon}. As in the ungraded case we still have $(\alpha_1 \otimes \alpha_2) \circ (\beta_1 \otimes \beta_2) = (\alpha_1 \circ \beta_1) \otimes (\alpha_2 \circ \beta_2)$. Therefore $\Ggr_D$ is an EH-$\mI$-group in the sense of Def.~\cite[Def.~3.1]{DP2}

\begin{lemma} \label{lem:Ggr_stable}
	Let $D$ be a strongly self-absorbing $C^*$-algebra. The EH-$\mI$-group $\Ggr_D$ is stable, ie.\ the functor $\Ggr_D$ maps all morphisms $\mathbf{m} \to \mathbf{n}$ with $m > 0$ in $\mI$ to homotopy equivalences.
\end{lemma}
\begin{proof} 
{Since the maps $\Ggr_D(\mathbf{n} \to \mathbf{n})$ are homeomorphisms, the statement is true for $m = n$. Hence, it suffices to prove it for $m = 1$, $n>1$ and the morphism $\iota_n \colon \mathbf{1} \to \mathbf{n}$ with $\iota(1) = n$.
Let $A=D \otimes \K$ and
 $j_n=\Ggr_D(\iota_n):\Autgr(\Cliff{1} \otimes A)\to \Autgr((\Cliff{1} \otimes A)^{\otimes n})$ . We shall write $j_n$ as the composition of three maps each of which is a homotopy equivalence. Let
\[
	\Phi \colon \mathrm{Aut}(\Cliff{1})\times \mathrm{Aut}(A)\cong \Autgr(\Cliff{1} \otimes A) \to \Autgr(\Cliff{1} \otimes A^{\otimes n})
\cong \mathrm{Aut}(\Cliff{1})\times \mathrm{Aut}(A^{\otimes n})\]
be the map $\Phi(\beta \otimes \alpha)=\beta \otimes \mathrm{id}_{A^{\otimes (n-1)}}\otimes \alpha$, $\beta\in \mathrm{Aut}(\Cliff{1})$, $\alpha\in \mathrm{Aut}(A)$. As explained in the proof of \cite[Thm.~4.5]{DP2}, the map $\mathrm{Aut}(A)\to \mathrm{Aut}(A^{\otimes n}),$
$\alpha \mapsto \mathrm{id}_{A^{\otimes (n-1)}}\otimes \alpha$ is a homotopy equivalence, therefore $\Phi$ is as well.
Let
\[J:\Autgr(\Cliff{1} \otimes A^{\otimes n}) \to \Autgr(\Cliff{1}^{\,\otimes n} \otimes A^{\otimes n}), \quad J(\gamma)= \mathrm{id}_{\Cliff{1}^{\,\otimes (n-1)}} \otimes \gamma\]
 be the homotopy equivalence given Proposition~\ref{lem:Cliff_hom_eq}.
The permutation homomorphism $\Cliff{1}^{\,\otimes n} \otimes A^{\otimes n} \to (\Cliff{1} \otimes A)^{\otimes n}$ defined by
\[(x_1\otimes \cdots \otimes x_n)\otimes (a_1\otimes  \cdots \otimes a_n) \mapsto (x_1\otimes a_1)  \otimes \cdots \otimes  (x_n\otimes a_n)\]
induces an isomorphism of topological groups
\[	
	\Psi \colon \Autgr(\Cliff{1}^{\,\otimes n} \otimes A^{\otimes n}) \to\Autgr((\Cliff{1} \otimes A)^{\otimes n}).
\]
Then $j_n=\Psi\circ J \circ \Phi$. Indeed
\[\Psi(J(\Phi(\beta \otimes \alpha)))=\Psi(J(\beta \otimes \mathrm{id}_{A^{\otimes (n-1)}}\otimes \alpha))=\Psi(\mathrm{id}_{\Cliff{1}^{\,\otimes (n-1)}}\otimes\beta \otimes \mathrm{id}_{A^{\otimes (n-1)}}\otimes \alpha)=j_n(\beta \otimes \alpha).
\]
This completes the proof.
}
\end{proof}

\begin{lemma} \label{lem:comp_inv}
	Let $D$ be a strongly self-absorbing $C^*$-algebra. For each $n \in \N$ the group $\pi_0(\Ggr_D(\mathbf{n}))$ is abelian and the EH-$\mI$-group $\Ggr_D$ has compatible inverses in the sense of \cite[Def.~3.1]{DP2}.
\end{lemma}

\begin{proof}
	Note that the second statement is a consequence of the first, because it implies that for a given $\alpha \in \Ggr_D(\mathbf{n})$ there is a path that connects $\id{} \otimes \alpha \in  \Ggr_D(\mathbf{n} \sqcup \mathbf{n})$ to
	\[
		\sigma \circ (\id{} \otimes \alpha) \circ \sigma^{-1} = \alpha \otimes \id{}\ ,
	\]
	where $\sigma \in \Aut((\Cliff{1} \otimes D \otimes \K)^{\otimes n} \otimes (\Cliff{1} \otimes D \otimes \K)^{\otimes n}) = \Ggr_D(\mathbf{n} \sqcup \mathbf{n})$ is the graded permutation of the first $n$ and the last $n$ tensor factors.
	
	By the periodicity of Clifford algebras it suffices to show the first statement for $n \in \{1,2\}$. By Lem.~\ref{lem:Autgr_Cl1_short_exact} and Lem.~\ref{lem:Autgr_Cl2_equiv} there are group isomorphisms
	\[
		\pi_0(\Ggr_D(\mathbf{n})) \cong \Z/2 \times \pi_0(\Aut(D \otimes \K)) \cong \Z/2 \times K_0(D)^\times_+
	\]
	in these cases (see \cite[Thm.~2.18]{DP1}) and the right hand side is abelian.
\end{proof}

\subsubsection{The action of $\Ggr_D$ on $KU^D$}
We can now define an action $\kappa_n^{\rm gr} \colon \Ggr_D(\mathbf{n}) \times KU^D_n \to KU^D_n$ of $\Ggr_D$ on the symmetric ring spectrum $KU^D$ as follows:
\[
	\kappa^{\rm gr}_n(\alpha, \varphi) = \alpha \circ \varphi\ .
\]
We will see in Thm.~\ref{thm:action_on_KUD} that this map indeed defines an action. Let $G_D$ be the EH-$\mI$-group from \cite[Sec.~4.2]{DP2} and let
\[
	\theta_n \colon (\Cliff{1} \otimes D \otimes \K)^{\otimes n} \to \Cliff{n} \otimes (D \otimes \K)^{\otimes n}
\]
be the unique $*$-isomorphism preserving the order of the tensor factors and identifying $(\Cliff{1})^{\otimes n}$ with $\Cliff{n}$. The maps $\psi_n \colon G_D(\mathbf{n}) \to \Ggr(\mathbf{n})$ defined by $\alpha \mapsto \theta_n^{-1} \circ (\id{\Cliff{n}} \otimes \alpha) \circ \theta_n$ form a morphism of EH-$\mI$-groups such that the following diagram commutes
\[
	\begin{tikzcd}
		G_D(\mathbf{n}) \times KU^D_n \ar[r,"\kappa_n"] \ar[d,"\psi_n \times \id{}" left] & KU^D_n \\
		\Ggr_D(\mathbf{n}) \times KU^D_n \ar[ur,"\kappa^{\rm gr}_n" below]
	\end{tikzcd}
\]
\subsection{The main technical result}
We now show the analogue of Theorem~\ref{thm:A} holds for the action of $\Ggr_D$ on $KU^D$ and will discuss some of its consequences.
\begin{theorem} \label{thm:action_on_KUD}
	Let $D$ be a strongly self-absorbing $C^*$-algebra. The EH-$\mI$-group $\Ggr_D$ acts via $\kappa^{\rm gr}$ on the ring spectrum $KU^D$ in the sense of \cite[Def.~3.7]{DP2}. The induced map
	\begin{equation} \label{eqn:map_of_Gamma}
		\Gamma(\Ggr_D) \to \Gamma(\Omega^\infty(KU^D)^*)
	\end{equation}
	of the associated $\Gamma$-spaces induces an isomorphism on all homotopy groups $\pi_n$ with $n > 0$ of the corresponding connective spectra and the homomorphism
	\[
		\{\pm 1\} \times K_0(D)_{+}^\times \to K_0(D)^\times \quad , \quad (a,b) \mapsto a \cdot b
	\]
	on $\pi_0$. In particular, \eqref{eqn:map_of_Gamma} induces an equivalence in the stable homotopy category of spectra if $D$ is stably finite and satisfies the UCT.
\end{theorem}

\begin{proof}
	The map $\kappa_n^{\rm gr}$ satisfies \cite[Def.~3.7 (i)]{DP2}. 	Let $\varphi \in KU^D_n$, $\psi \in KU^D_m$, $\alpha \in \Ggr_D(\mathbf{n})$ and $\beta \in \Ggr_D(\mathbf{m})$. We have
	\(
		\left(	(\alpha \circ \varphi) \otimes (\beta \circ \psi)\right) \circ \Delta = (\alpha \otimes \beta) \circ (\varphi \otimes \psi) \circ \Delta
	\),
	which shows that the diagram in \cite[Def.~3.7 (ii)]{DP2} commutes. With $\varphi \in KU^D_m$ and $\eta_l \in \Omega^l KU^D_l$ as in Sec.~\ref{sec:spectra_and_units} we have $\eta_l \ast (\alpha \circ \varphi) = (\id{} \otimes \alpha) \circ (\eta_l \ast \varphi) \in \Omega^lKU_D^{l + m}$. This proves that \cite[Def.~3.7 (iii)]{DP2} also commutes. Hence, $\kappa^{\rm gr}$ indeed defines an action of $\Ggr_D$ on $KU^D$. The map in \eqref{eqn:map_of_Gamma} is induced by the morphism of $\mI$-monoids $\Ggr_D \to \Omega^{\infty}(KU^D)^*$ that exists by Lem.~\ref{lem:comp_inv} and \cite[Thm.~3.8]{DP2}.
	
	To understand the induced maps on homotopy groups note that $\psi_n \colon G_D(\mathbf{n}) \to \Ggr_D(\mathbf{n})$ induces a homotopy equivalence on $\pi_n$ for all $n > 0$ by Lem.~\ref{lem:Autgr_Cl1_short_exact}, Lem.~\ref{lem:Autgr_Cl2_equiv} and Proposition~\ref{lem:Cliff_hom_eq}. Moreover, the following diagram of $\mI$-spaces commutes:
	\[
	\begin{tikzcd}
		G_D \ar[r] \ar[d] & \Omega^\infty(KU^D)^* \\
		\Ggr_D \ar[ur]
	\end{tikzcd}
	\]
	Since $G_D$ as well as $\Ggr_D$ are both stable, the induced maps on the homotopy colimits and therefore also on the $\Gamma$-spaces give isomorphisms on all $\pi_n$ for $n > 0$. By stability it suffices to show that $\psi_1$ induces the homomorphism from the statement on $\pi_0$, ie.\ that
	\[
		\psi_{1*}\colon \pi_0(\Autgr(\Cliff{1} \otimes D \otimes \K)) \to \pi_0(\Omega (KU^D_1)^*)
	\]
	is the desired map. By Lem.~\ref{lem:Autgr_Cl1_short_exact} the left hand side is isomorphic to $\Z/2 \times \pi_0(\Aut(D \otimes \K))$. The restriction of $\psi_{1*}$ to $\{1\} \times \pi_0(\Aut(D \otimes \K))$ agrees with the inclusion $K_0(D)_+^\times \to K_0(D)^\times$ under the identifications $\pi_0(\Aut(D \otimes \K)) \cong K_0(D)_+^\times$ and $K_0(D)^\times \cong \pi_0(\Omega (KU^D_1)^*) \cong \pi_1((KU^D_1)^*)$ by \cite[Thm.~4.6]{DP2}. The non-trivial automorphism $\theta \in \Aut(\Cliff{1})$ acts by $\theta(e_1) = -e_1$ on the odd degree generator $e_1 \in \Cliff{1}$. Let $\eta_1 \in \Omega KU^D_1$ be as above and denote by $i \colon S^1 \to S^1$ the map reversing the orientation of the circle, ie.\ $i(z) = \overline{z}$. Since $\theta \circ \eta_1 = \eta_1 \circ i \colon S^1 \to KU^D_1$, the element $[\theta \otimes \id{D \otimes \K}] \in \pi_0(\Autgr(\Cliff{1} \otimes D \otimes \K))$ is mapped to $(-1) \in K_0(D)^\times$. Since the morphism of $\mI$-monoids $\Ggr_D \to \Omega^{\infty}(KU^D)^*$ gives rise to a group homomorphism on $\pi_0$ the proof is finished. 	
\end{proof}
It is convenient to introduce the commutative $\mI$-monoid $\bar{G}_{D}$ that associates to $\mathbf{n}$ the path-components of the identity, i.e.\ $\bar{G}_{D}(\mathbf{n})=\Aut_0((D\otimes \mathcal{K})^{\otimes n})$. The maps are defined just like those of $G_D$.
\begin{definition}
The cohomology theories defined by the  $\Gamma$-spaces
$\Gamma(\bar{G}_{D})$, $\Gamma(G_{D})$, $\Gamma(\Ggr_{D})$ are denoted by $\bE_{D}^*(X)$, $E_{D}^*(X)$, and $\hE_{D}^*(X)$ respectively.
\end{definition}

The following corollary yields Theorem~\ref{thm3:intro} from the introduction.
Let $P$ be a set of prime numbers.
\begin{corollary}\label{Cor:crucial}
	The following maps of $\Gamma$-spaces are all equivalences:
	\begin{align*}
		\Gamma(\Ggr_\ZZ) \to \Gamma(\Omega^\infty(KU^\ZZ)^*) \to \Gamma(\Omega^\infty(KU^{\OO_\infty})^*) \leftarrow \Gamma(G_{\OO_\infty})
	\end{align*}
	\begin{align*}
		\Gamma(\Ggr_{\ZZ\otimes M_P}) \to \Gamma(\Omega^\infty(KU^{\ZZ\otimes M_P})^*) \to \Gamma(\Omega^\infty(KU^{\OO_\infty \otimes M_P})^*) \leftarrow \Gamma(G_{\OO_\infty\otimes M_P})
	\end{align*}
In particular, for any $n \in \N$ these equivalences induce natural group isomorphisms
\[
		\hE_{\ZZ}^*(X)\cong E_{\OO_\infty}^*(X)\quad,\quad 	
		\hE_{\ZZ\otimes M_P}^*(X)\cong E_{\OO_\infty\otimes M_P}^*(X).
	\]
	for any finite CW-complex $X$.
\end{corollary}
\begin{proof} If $P=\emptyset$ we use the convention $M_P=\C$. The first line of equivalences then follows from the second. Applying Thm.~\ref{thm:action_on_KUD} in the case $D = \ZZ\otimes M_P$ shows that the first map is an equivalence. Theorem~\ref{thm:A} applied for $D=\OO_{\infty}\otimes M_P$ gives the third equivalence.

	The  equivalence in the middle was shown in \cite[Thm.~4.7]{DP2} for $M_P=\C$. More generally, as argued earlier in the proof of Thm.~\ref{thm:B},
	the arguments from the proof of \cite[Thm.4.7]{DP2} show that the KK-equivalence $\ZZ\otimes M_P \to \OO_{\infty}\otimes M_P$  induces an equivalence
	\[
	\Gamma(\Omega^\infty(KU^{\ZZ\otimes M_P})^*) \to \Gamma(\Omega^\infty(KU^{\OO_\infty \otimes M_P})^*). \qedhere
	\]
\end{proof}

\begin{corollary}\label{Rem-crucial}
{The cohomology theories $\hE_{\ZZ}^1(X)$ and $\hE_{M_P}^1(X)$ classify principal bundles with structure groups $\Autgr(\Cliff{n} \otimes \ZZ \otimes \K)$ and $\Autgr(\Cliff{n} \otimes  M_P\otimes \K)$, respectively. Furthermore, for $P\neq \emptyset$ and any $n\geq 1$ we have natural isomorphisms}	
\[
		\hE_{\ZZ}^1(X) \cong [X,B\Autgr(\Cliff{n} \otimes \ZZ \otimes \K)] \cong [X,B\Aut(\OO_\infty \otimes \K)] \cong E_{\OO_\infty}^1(X)\ ,
	\]
	\[
		\hE_{M_P}^1(X) \cong [X,B\Autgr(\Cliff{n} \otimes  M_P\otimes \K)] \cong [X,B\Aut(\OO_\infty \otimes M_P\otimes \K)] \cong E_{\OO_\infty\otimes M_P}^1(X)\ .
	\]
\end{corollary}
\begin{proof}
By Lemma~\ref{lem:Ggr_stable} the $\mI$-monoids $\Ggr_{D}$ are stable. The same argument as in \cite{DP2} shows that the first group $\hE_D^1(X)$ of the associated cohomology theory $\hE_D^*(X)$ classifies bundles. We summarise the main steps: Since $\Ggr_{D}$ takes values in topological groups, the homotopy colimit $(\Ggr_{D})_{h\mI}$ is a topological group as well. Let $B(\Ggr_{D})_{h\mI}$ be its classifying space. Since the classifying space functor preserves products, $\mathbf{n} \mapsto B\!\Ggr_{D}(\mathbf{n})$ defines a commutative $\mI$-monoid with homotopy colimit $(B\!\Ggr_{D})_{h\mI}$. By \cite[Lem.~3.4]{DP2}, those two spaces are homotopy equivalent, i.e.\ $(B\!\Ggr_{D})_{h\mI} \simeq B(\Ggr_{D})_{h\mI}$. Hence, $(B\!\Ggr_{D})_{h\mI}$ is a delooping of $(\Ggr_{D})_{h\mI}$. Note that the $\mI$-monoid multiplication of $\mathbf{n} \mapsto B\!\Ggr_{D}(\mathbf{n})$ encodes the tensor product of bundles.

There is another space delooping $(\Ggr_{D})_{h\mI}$. It is the one obtained from the infinite loop space structure (the $\Gamma$-space delooping), and we will denote it by $B_\otimes(\Ggr_{D})_{h\mI}$. By definition of the cohomology theory, $\hE_{D}^1(X) = [X, B_\otimes(\Ggr_{D})_{h\mI}]$. A priori this space could be different from $B(\Ggr_{D})_{h\mI}$, but because $\Ggr_{D}$ is a stable EH-$\mI$-group, \cite[Thm.~3.6]{DP2} implies
\[
 	B\Autgr(\Cliff{1} \otimes D \otimes \K) = B\!\Ggr_{D}(\mathbf{1}) \simeq B(\Ggr_{D})_{h\mI} \simeq B_\otimes(\Ggr_{D})_{h\mI}\ .
\]
and therefore $\hE_D^1(X) \cong [X, B\Autgr(\Cliff{1} \otimes D \otimes \K)]$, where the group structure of $\hE_D^1(X)$ is identified with the tensor product of bundles by the observations above. The group homomorphism $ (\id{\Cliff{n-1}}\otimes\, \cdot ) \colon  \Autgr(\Cliff{1} \otimes D \otimes \K) \to \Autgr(\Cliff{n} \otimes D \otimes \K)$ is an equivalence for any $n \in \N$. Therefore $B\Autgr(\Cliff{1} \otimes D \otimes \K) \simeq B\Autgr(\Cliff{n} \otimes D \otimes \K)$. This shows $\hE_{D}^1(X) \cong [X,B\Autgr(\Cliff{n} \otimes D \otimes \K)]$, for all $n \in \N$. The equivalences of $\Gamma$-spaces established in Cor.~\ref{Cor:crucial} in particular give homotopy equivalences
\[
	B_\otimes(\Ggr_{\ZZ})_{h\mI} \simeq B_\otimes(\G_{\OO_{\infty}})_{h\mI} \qquad \text{and} \qquad B_\otimes(\Ggr_{M_P})_{h\mI} \simeq B_\otimes(\G_{\OO_{\infty} \otimes M_P})_{h\mI}\ .
\]
Hence, we obtain natural isomorphisms $\hE^1_{\ZZ}(X) \cong E^1_{\OO_\infty}(X)$ and $\hE^1_{M_P}(X) \cong E^1_{ \OO_\infty \otimes M_P}(X)$. The stability of the $\mI$-monoids $\G_{\OO_\infty}$ and $\G_{\OO_\infty \otimes M_P}$ was established in \cite{DP2} showing that the groups $E^1_{\OO_\infty}(X)$ and $E^1_{\OO_\infty \otimes M_P}(X)$ also classify bundles.
\end{proof}
\begin{corollary}\label{rem-cliff-stab}
  Suppose that $A_1,A_2$ are locally trivial bundles of graded $C^*$-algebras with all fibers isomorphic to $\Cliff{n_i}\otimes D \otimes \mathcal{K}$ and structure groups $\Autgr(\Cliff{n_i}\otimes D \otimes \mathcal{K})$, $i=1,2$. Then  $A_1$ is isomorphic to
$A_2$  if and only if $[{A}_1]=[{A}_2]$ in $\hE_{D}^1(X)$  and $n_1\equiv n_2$ $(\mathrm{mod}$ $ 2)$.
\end{corollary}
\begin{proof}
Indeed, by stability of $\Ggr_D$, the inclusion $\Ggr_D(\mathbf{n})\to \Ggr_D(\mathbf{n+1})$,
$\alpha\mapsto \mathrm{id}_{\Cliff{1}}\otimes \alpha$ induces a bijection
\[[X,B\Autgr(\Cliff{n}\otimes D \otimes \mathcal{K})]\to [X,B\Autgr(\Cliff{n+1}\otimes D \otimes \mathcal{K})].\]
At the bundle level this corresponds to the map $\mathcal{A} \mapsto \Cliff{1}\otimes \mathcal{A}$.
\end{proof}

The following is the same as Theorem~\ref{thm1:intro} from the introduction.
\begin{theorem}\label{thm:neww}
The tensor product operation defines a group structure on the isomorphism classes of  locally trivial bundles of graded $C^*$-algebras with fibers isomorphic to $\Cliff{k}\otimes D \otimes \mathcal{K}$, $k\geq 1$ variable. This group is
isomorphic to $H^0(X,\Z/2)\times \hE_{D}^1(X)$.
  \end{theorem}
\begin{proof}
Let $A$ be a bundle as in the statement with fibers $A(x)=\Cliff{k(x)}\otimes D \otimes \mathcal{K}$.
The map $x\mapsto k(x)$ is locally constant by assumption.
The reduction $(\mathrm{mod}$ $ 2)$ of $k(x)$ defines element of  $H^0(X,\Z/2)$ associated to  the isomorphism class of $A$. We conclude by applying Corollary~\ref{rem-cliff-stab} for each connected component of $X$.
\end{proof}

\section{Computation of $\hE_\C^1(X)$ and $\hE_{M_P}^1(X)$}
In this section we will (implicitly) present a new model for $BGL_1(KU^D)$ that is compatible with classical results about the Brauer group and Clifford algebras. As a space, $BGL_1(KU)$ decomposes as a product
\[
	BGL_1(KU) \simeq K(\Z/2,1) \times K(\Z,3) \times BBSU_{\otimes}\ .
\]
However, this decomposition does not respect the infinite loop space structure. In fact, the first two factors $K(\Z/2,1) \times K(\Z,3)$ split off from the rest and the group structure induced on homotopy classes, ie.\ on the set $H^1(X,\Z/2) \times H^3(X,\Z)$, is
\[
	(w, \tau) \cdot (w',\tau') = (w + w', \tau + \tau' + \beta(w \cup w'))
\]
for $w,w' \in H^1(X,\Z/2)$ and $\tau,\tau' \in H^3(X,\Z)$, where $\beta \colon H^2(X,\Z/2) \to H^3(X,\Z)$ is the Bockstein homomorphism. As explained in \cite{paper:DonovanKaroubi} the appearance of the Bockstein homomorphism is closely linked to the behaviour of Clifford bundles with respect to the fibrewise tensor product.

\subsection{Comparing $E^1_{M_P}(X)$ with $\hE^1_{M_P}(X)$.}
Let $H$ be a multiplicative commutative group with unit $1$. The constant $\mI$-monoid $H$ associated to this group has
 objects $H(\mathbf{n})=H$ if $n>0$ while $H(\mathbf{0})=\{1\}$. The morphisms corresponding to $\mathbf{m}\to \mathbf{n}$ are identity maps for $m>0$. The multiplication $H(\mathbf{m})\times H(\mathbf{n})\to H(\mathbf{m\sqcup n})$ is the group multiplication on $H$.

\begin{proposition}\label{prop:pi0}
Let $\G$ be a stable Eckmann-Hilton $\mathcal{I}$-group. Suppose that each permutation $\mathbf{n}\to \mathbf{n}$ induces the identity map on $\pi_0\G(\mathbf{n})$. Then the group $\pi_0\G(\mathbf{1})$ is abelian and
 there is canonical map of $\mI$-monoids $\theta:\G\to \pi_0\G(\mathbf{1}).$
\end{proposition}
\begin{proof}
Since each morphism $\mathbf{n}\to \mathbf{n}$ induces the identity map on $\pi_0\G(\mathbf{n})$,
it follows that all morphisms $\mathbf{1}\to \mathbf{n}$ induce the same map $j_n:\pi_0\G(\mathbf{1})\to \pi_0\G(\mathbf{n})$. Since $\G$ is stable, $j_n$ is an isomorphism of groups.
Define $\theta_n:\G(\mathbf{n})\to \pi_0\G(\mathbf{1})$ as the composition
\[
\begin{tikzcd}
	\G(\mathbf{n}) \ar[r] & \pi_0\G(\mathbf{n})  \ar[r ,"j_n^{-1}"] & \pi_0\G(\mathbf{1}),
	\end{tikzcd}
\]
It follows that for all $m>0$ and all morphisms $f:\mathbf{m}\to \mathbf{n}$  the diagram
\[	\begin{tikzcd}
			\pi_0\G(\mathbf{m}) \ar[r, "\pi_0f"] & \pi_0\G(\mathbf{n}) \\
 & \pi_0\G(\mathbf{1})\ar[ul, "j_m"  ]\ar[u, "j_n" swap]
		\end{tikzcd}
\]
is commutative. This shows that $\theta$ is a map of $\mI$-spaces. The transformation $\theta$ preserves multiplication if  the following diagram is commutative:
\begin{equation}\label{diagram:mu}
		\begin{tikzcd}
			\pi_0\G(\mathbf{m}) \times \pi_0\G(\mathbf{n})\ar[r] & \pi_0\G(\mathbf{m \sqcup n}) \\
\pi_0\G(\mathbf{1}) \times \pi_0\G(\mathbf{1})\arrow[u, shift left=4.5ex  ]\arrow[u,  shift right=5.5ex]\ar[r] & \pi_0\G(\mathbf{1})\ar[u]
		\end{tikzcd}
	\end{equation}
By naturality of $\mu_{m,n}$ the diagram
\[
		\begin{tikzcd}
			\G(\mathbf{m}) \times \G(\mathbf{n})\ar[r] & \G(\mathbf{m\sqcup n}) \\
\G(\mathbf{1}) \times \G(\mathbf{1})\arrow[u, shift left=4.5ex  ]\arrow[u,  shift right=5.5ex]\ar[r] & \G(\mathbf{1\sqcup 1})\ar[u]
		\end{tikzcd}
	\]
is commutative and hence we may assume that $m=n=1$ in \eqref{diagram:mu}. For simplicity we will write $gh$ for $\nu_n(g,h)$, $g,h \in \G(\mathbf{n})$. The class of $g$ in $\pi_0\G(\mathbf{n})$ is denoted by $[g]$.

Our task is to verify that
\[j_{2}[gh]=[\mu_{1,1}(g,h)].\]
As noted in the third line of the proof of \cite[Lemma 3.2]{DP2}, by the naturality properties of $\mu_{m,n}$ it follows that $\mu_{1,1}(1,g)=(\iota_{2})_*(g)$, where $\iota_2:\mathbf{1}\to \mathbf{2}$, $\iota_2(1)=2$. From this we see that $[\mu_{1,1}(1,g)]=j_2[g]$ for $g\in \G(\mathbf{1})$.
Thus we need to show that
\[[\mu_{1,1}(1,gh)]=[\mu_{1,1}(g,h)].\]
The Eckmann-Hilton property requires that
$\mu_{1,1}(g,g')\mu_{1,1}(h,h')=\mu_{1,1}(gh,g'h')$.
It follows that
\[\mu_{1,1}(1,gh)=\mu_{1,1}(1,g)\mu_{1,1}(1,h)\]
and
\[\mu_{1,1}(g,h)=\mu_{1,1}(g,1)\mu_{1,1}(1,h)=(\tau_{1,1})_*(\mu_{1,1}(1,g))\mu_{1,1}(1,h),\]
since any Eckmann-Hilton $\mI$-monoid is commutative by \cite[Lemma 3.2]{DP2}.
The $\pi_0$ class of the block permutation $(\tau_{1,1})_*$ equals the identity map (by hypothesis). It follows  that
\[[\mu_{1,1}(g,h)]=[\mu_{1,1}(1,g)][\mu_{1,1}(1,h)]=[\mu_{1,1}(1,gh)].\]
Since
\[j_{2}[gh]=[\mu_{1,1}(g,h)]=[(\tau_{1,1})_*(\mu_{1,1}(h,g)]=[\mu_{1,1}(h,g)]=j_{2}[hg],\]
the commutativity of $\pi_0\G(\mathbf{1})$ follows from stability.
\end{proof}

Let $D$ be a strongly selfabsorbing $C^*$-algebra.
  Fix the idempotent $e=1_D \otimes e_{11}\in D \otimes \K$ where $e_{11}$ is a rank one projection.  For each $n\geq 1$ fix an isomorphism $\theta_n:(D \otimes \K)^{\otimes n}\to D \otimes \K$ such that $\theta_n(e^{\otimes n})=e$. The space of all such isomorphisms in contractible by \cite{DP1}.
Define maps $\mathrm{Aut}((D \otimes \K)^{\otimes n})\to \KD$ by $\alpha\mapsto [\theta_n(\alpha(e^{\otimes n}))]$ for $n>0$ and
$\mathrm{Aut}(\C)=\{\mathrm{id}\} \to \{1\}$.
\begin{lemma}\label{lemma:K0}
  The maps $\G_D(\mathbf{n}) \to \KD$ define a morphism of $\mI$-monoids.
\end{lemma}
\begin{proof}
It was shown in \cite{DP1} that $\G_D$ is a stable $EH$-$\mI$-group and that $\pi_0\G_D(\mathbf{1})\cong \KD$, via the map $\alpha\mapsto [\alpha(e)]$.
 Any permutation $\mathbf{n}\to \mathbf{n}$ induce an automorphism $\sigma \in \mathrm{Aut}((D \otimes \K)^{\otimes n})$ which fixes $e^{\otimes n}$ and hence $\sigma$ is homotopic to the identity map \cite{DP1}.
  Thus $[\sigma \circ \alpha \circ \sigma^{-1}]=[\alpha]$ in  $\pi_0\G_D(\mathbf{n}),$
 for all $\alpha \in \G_D(\mathbf{n})$. The statement follows from now from Proposition~\ref{prop:pi0}.
\end{proof}

\begin{lemma}\label{lemma:K00}
  The maps $\Ggr_D(\mathbf{n}) \to \Z/2 \times \KD \to \Z/2$ define morphisms of $\mI$-monoids.
\end{lemma}
\begin{proof}
By Lemma~\ref{lem:Ggr_stable} the $EH$-$\mI$-group $\Ggr_D$ is stable and $\pi_0\Ggr_D(\mathbf{1})\cong \Z/2 \times \KD$.
Moreover we have shown in the proof of Lemma~\ref{lem:comp_inv}   that each permutation $\mathbf{n}\to \mathbf{n}$ induces the identity map on $\pi_0\Ggr_D(\mathbf{n})$. We conclude the proof by applying Proposition~\ref{prop:pi0}.
\end{proof}

\begin{theorem}\label{thm:Z2}
	Let $X$ be a finite CW-complex and let $D$ be a stable finite strongly self-absorbing $C^*$-algebra satisfying the UCT. The groups $E^1_{D}(X)$ and $\hE^1_{D}(X)$ fit into a short exact sequence
	\[
	\begin{tikzcd}
		0 \ar[r] & E^1_{D}(X) \ar[r] & \hE^1_{D}(X) \ar[r ,"\delta_0"] & H^1(X, \Z/2) \ar[r] & 0.
	\end{tikzcd}
	\]
	If $L$ is a real line bundle on $X$ with associated Clifford bundle $\Cliff{L}$, then $\delta_0(\Cliff{L}\otimes D \otimes \mathcal{K})=w_1(L)$, where $w_1(L)$ is the first Stiefel-Whitney class of $L$.
\end{theorem}

\begin{proof} 
	There are maps of $\mI$-monoids:
\[G_D(\mathbf{n}) \to \Ggr_D(\mathbf{n}) \to \pi_0(\Ggr_D(\mathbf{n}))/\pi_0(G_D(\mathbf{n}))\cong \Z/2.\]
The first map was discussed earlier in this section. The second map sends the $\pi_0$-class of the image of the grading homomorphism $\nu \otimes \id{D \otimes \mathcal{K}}$ in $\Ggr_D(\mathbf{n})$ to the generator of $\Z/2$. If $c_{ij} \colon U_{ij} \to \Z/2$ is the cocycle representing $w_1(L)$ for the line bundle $L \to X$, then $\Cliff{L}$ is the algebra bundle represented by $d_{ij} \colon U_{ij} \to \Aut(\Cliff{1})$ obtained by composing $c_{ij}$ with the isomorphism $\Z/2 \to \Aut(\Cliff{1})$ that maps the generator of $\Z/2$ to $a + b\,e \mapsto a - b\,e$ for $a,b \in \C$ and $e \in \Cliff{1}$ with $e^2 = 1$. But this is the homomorphism that is used in the definition of $\delta_0$ to map $H^1(X,\pi_0(\Aut(\Cliff{1} \otimes D \otimes \K)))$ to $H^1(X, \Z/2)$. Hence, we have $\delta_0(\Cliff{L} \otimes D \otimes \K) = w_1(L)$.
 
The isomorphism of groups
\[\Z/2\times \Aut(D\otimes \mathcal{K}) =\Autgr(\Cliff{1} )\times \Aut(D\otimes \mathcal{K}) \to \Autgr(\Cliff{1} \otimes D\otimes \mathcal{K})\]
from {Lemma}~\ref{lem:Autgr_Cl1_short_exact} induces a homotopy equivalence
\[B(\Z/2 )\times B\Aut(D\otimes \mathcal{K}) \to B\Autgr(\Cliff{1} \otimes D\otimes \mathcal{K})\]
and hence a bijection
\begin{equation}\label{eqn:splito}
H^1(X, \Z/2)\times E^1_{D}(X) \to \hE^1_{D}(X),
\end{equation}
whose restriction to the second component induces a morphism of groups $E^1_{D}(X) \to \hE^1_{D}(X)$.
\end{proof}
\begin{remark}\label{remark:review} We shall use several times the following basic fact, \cite[p.93]{Brown:book-cohomology}.
Suppose that
\[
	\begin{tikzcd}
		0 \ar[r] & A \ar[r,"i"] & E \ar[r,"\pi"] & G \ar[r] & 0
	\end{tikzcd}
	\]
	is an extension of abelian groups and $\sigma:G \to E$ is a  map such that $\sigma(0)=0$ and $\pi\circ \sigma=\mathrm{id}_E$. Let $c:G\times G \to A$ be the normalized $2$-cocycle defined by
	$i(c(g,h))=\sigma(gh)\sigma(h)^{-1}\sigma^{-1}(g)$, $g,h\in G$. Then the group $E$ is isomorphic to $G\times A$ endowed with the group law:
	\[(g,a)(g',a')=(g+g',a+a'+c(g,g')).\]
\end{remark}
\begin{proposition}[Cf. \cite{paper:DonovanKaroubi}, \cite{paper:Parker-Brauer}]\label{prop:Z2}
For any finite CW-complex $X$, $\hE_\C^1(X)\cong H^1(X,\Z/2) \times_{_{tw}}  H^3(X,\Z)$ with group structure:
\[
	(w, \tau) \cdot (w',\tau') = (w + w', \tau + \tau' + \beta(w \cup w'))
\]
for $w,w' \in H^1(X,\Z/2)$ and $\tau,\tau' \in H^3(X,\Z)$, where $\beta \colon H^2(X,\Z/2) \to H^3(X,\Z)$ is the Bockstein homomorphism.
\end{proposition}
\begin{proof}
By Theorem~\ref{thm:Z2} there is an extension
\[
	\begin{tikzcd}
		0 \ar[r] & E^1_{\C}(X) \ar[r] & \hE^1_{\C}(X) \ar[r] & H^1(X, \Z/2) \ar[r] & 0
	\end{tikzcd}
	\]
	and we know already that $E^1_{\C}(X) \cong H^3(X,\Z)$ via the Dixmier-Douady invariant $\delta$. To complete the proof we show that
	a normalized 2-cocycle $c:H^1(X, \Z/2)\times H^1(X, \Z/2) \to H^3(X,\Z)$ for the extension above is given by
	$c(w,w')=\beta(w \cdot w')$ where ``$\cdot$" stands the for the cup-product.
	We will need the remark that  $\beta (x^2)=0$  for all $x\in H^1(X, \Z/2)$.
	Indeed, since $x\mapsto \beta (x^2)$ is a natural operation $H^1(X, \Z/2) \to H^3(X,\Z),$ it must be induced by a map $K(\Z/2,1)\to K(\Z,3)$. But all such maps are null-homotopic since
	 \(H^3(K(\Z/2,1),\Z)=H^3(\R\mathbb{P}^\infty,\Z)=0.\)
	 For $w,w'\in H^1(X, \Z/2),$ let $L,L'$ be real line bundles such that $w_1(L)=w$ and  $w_1(L')=w'$.
	 Then $w_1(L\otimes L')=w+w'$.
	 By Theorem~\ref{thm:Z2}, $[L]\mapsto [\Cliff{L}]$ is a section of $\delta_0$.
	 Therefore a normalized 2-cocycle for the extension above is given by
	 \[ c(w,w')=[\Cliff{L}{\otimes}\Cliff {L'}\otimes \Cliff{L\otimes L'}]=[\Cliff{L\oplus L'\oplus L\otimes L'}].\]
	 $[\Cliff{L\oplus L'\oplus L\otimes L'}]=[\Cliff{L\oplus L'\oplus L\otimes L' \oplus 1}]$ in $\hE^1(X)$,
	  where $1$ denotes the trivial line bundle.
	 The vector bundle $V=L\oplus L'\oplus L\otimes L' \oplus 1$ is orientable since $w_1(V)=w_1(L)+w_1(L')+w_1(L\otimes L' )=2w_1(L)+2w_1(L')=0$.
	 For an oriented real vector  bundle $V$ of even rank it is known that the Dixmier-Douady class
	of the complex Clifford bundle associated to $V$ satisfies
	$\delta(\Cliff{V})=\beta(w_2(V))$ where $w_2(V)\in H^2(X, \Z/2)$ is
	the second Stieffel-Whitney class of $V$, \cite[Thm.2.8]{Plymen}.
		One computes  $w_2(L\oplus L' \oplus L\otimes L' \oplus 1)=w_2(L\oplus L')+w_1(L\oplus L')w_1(L\otimes L' )=w_1(L)\cdot w_1(L')+(w_1(L)+w_1(L'))^2$.
	 It follows that
	 \[c(w,w')=\delta(\Cliff{V}) =\beta(w_2(V))=\beta(w_1(L)\cdot w_1(L'))=\beta(w \cdot w').\]	 	
We conclude the proof by applying Remark~\ref{remark:review}.\end{proof}

The following is the same as Theorem~\ref{thm2:intro} from the introduction.
\begin{theorem}\label{thm:basic}
Let $X$ be a finite CW-complex and let $D$ be a stably finite strongly self-absorbing $C^*$-algebra satisfying the UCT. There is an isomorphism of groups
\[\hE^1_{D}(X)\cong   H^1(X;\Z/2) \times_{_{tw}} E^1_{D}(X)\]
with multiplication on the direct product  $H^1(X;\Z/2) \times E^1_{D}(X)$ given by
\[
	(w, \tau) \cdot (w',\tau') = (w + w', \tau + \tau' + j_P\circ\beta(w \cup w'))
\]
for $w,w' \in H^1(X,\Z/2)$ and $\tau,\tau' \in E^1_{D}(X)$, where $j_P \colon   E^1_{\C}(X)\to E^1_{D}(X)$ is
the map induced by the unital $*$-homomorphism $\C\to D$ and we identify
$E^1_{\C}(X)\cong H^3(X,\Z)$.
\end{theorem}
\begin{proof}
The first isomorphism follows from Corollary~\ref{Cor:crucial}.
If $D'\to D$ is a unital $*$-monomorphism of strongly self-absorbing $C^*$-algebras, by Lemma~\ref{lemma:K00}
there is a commutative diagram of commutative  $\mI$-monoids:
\[
		\begin{tikzcd}
			G_D(\mathbf{n}) \ar[r] & \Ggr_D(\mathbf{n}) \ar[r] &  \Z/2\\
		G_{D'}(\mathbf{n}) \ar[r]\ar[u] & \Ggr_{D'}(\mathbf{n}) \ar[r] \ar[u]&  \Z/2\ar[equal]{u}
		\end{tikzcd}
	\]
which induces a commutative diagram
	\[
		\begin{tikzcd}
		0 \ar[r]&	E^1_{D}(X) \ar[r] & \hE^1_{D}(X) \ar[r] &  H^1(X;\Z/2) \ar[r] & 0\\
		0 \ar[r]& E^1_\C(X) \ar[r]\ar[u] & \hE^1_\C(X)  \ar[r] \ar[u]&   H^1(X;\Z/2)\ar[equal]{u}\ar[r] & 0\\
		0 \ar[r]& H^3(X;\Z) \ar[r]\ar[equal]{u} & H^3(X;\Z)\times_{_{tw}}   H^1(X;\Z/2) \ar[r] \ar[equal]{u}&   H^1(X;\Z/2)\ar[equal]{u}\ar[r] & 0
		\end{tikzcd}
	\]

Let  $j: H^3(X,\Z)= E^1_{\C}(X)\to E^1_{D}(X)$ be the map induced by the unital $*$-homomorphism $\C \to D$.
From the diagram above, Proposition~\ref{prop:Z2} and Remark~\ref{remark:review}  we obtain that
\begin{equation}\label{eqn:ZZZP}
  \hE^1_{D}(X)\cong H^1(X;\Z/2) \times_{_{tw}}   E^1_{D}(X)
\end{equation}
where the group structure is given by $(w, x) \cdot (w', x') = (w + w', x + x' + j(\beta(w \cup w')))$
for $w,w' \in H^1(X,\Z/2)$ and $x,x' \in E^1_{D}(X)$. Note that the image of $j$ is contained in $\bE^1_{D}(X)$
since $E^1_{\C}(X)=\bE^1_{\C}(X)$.
\end{proof}

\begin{corollary} Let $X$ be a finite CW complex of dimension $\leq 4$ and let $P$ be a nonempty set of primes. Then $E^1_{O_\infty}(X)\cong \hE_\C^1(X)$. More generally
\[E^1_{M_P\otimes O_\infty}(X)\cong H^1(X,\Z/2) \times_{_{tw}}  H^3(X,\Z_P)\] with group structure:
\[
	(w, \tau) \cdot (w',\tau') = (w + w', \tau + \tau' + \beta_P(w \cup w'))
\]
for $w,w' \in H^1(X,\Z/2)$ and $\tau,\tau' \in H^3(X,\Z_P)$, where $\beta_P \colon H^2(X,\Z/2) \to H^3(X,\Z_P)$ is the composition of the  Bockstein homomorphism with the coefficient map $H^3(X,\Z)\to H^3(X,\Z_P)$.
\end{corollary}
We refer the reader to \cite{DMcCP} for computations of $E^1_{\ZZ}(X)$, $E^1_{M_P}(X)$ and  $E^1_{{\OO_\infty}\otimes M_P}(X)\cong \hE^1_{M_P}(X)$ for a general finite CW-complex.

\subsection{The Brauer group}\label{The Brauer group}
Let $F$ be a graded unital $C^*$-algebra and let $X$ be a connected
compact metrizable space. We denote by $\hat{\mathscr{C}}_{F}(X)$ the set of locally trivial continuous fields of graded $C^*$-algebras with fiber $F$. We will tacitly identify the isomorphism classes of such continuous fields with the isomorphism classes of locally trivial principal $\Autgr(F)$-bundles. Let $D$ be a strongly self-absorbing $C^*$-algebra.
A  graded continuous field of $C^*$-algebras is called negligible if it is isomorphic with
 $p\big(C(X, M_N ( D) \otimes \Cliff{r} )\big)p$ for some full projection $p\in C(X, M_N ( D) \otimes \Cliff{r} )  $ of  degree zero, for some $r,N\geq 1$. Due to the periodicity of the Clifford algebras,  there is no loss of generality if one assumes that  $r\in \{1,2\}$ in the definition of negligible $C^*$-algebras.
\begin{definition} \label{def:Brauer} The graded Brauer group $\hat{Br}_D(X)$ consists of
equivalence classes of continuous fields $A\in \bigcup_{n,k\geq 1}\mathscr{C}_{M_n(D)\otimes \Cliff{k}}(X)$.
Two continuous fields $A_i\in \mathscr{C}_{M_{n_i}(D)\otimes \Cliff{r_i}}(X)$, $i=1,2$ are equivalent, if
there is a graded $C(X)$-linear isomorphism $$A_1 \otimes C_1\cong A_2 \otimes C_2 ,$$
for some negligible continuous fields $C_1,C_2$. We denote by $[A]_{\hat{Br}}$ the class of $A$ in $\hat{Br}_D(X)$.
The multiplication on $\hat{Br}_D(X)$ is induced by the tensor product operation, after fixing an isomorphism
$D\otimes D\cong D$.  We will show in a moment that the monoid $\hat{Br}_D(X)$ is a group.
\end{definition}
If $B\in \hat{\mathscr{C}}_{ D \otimes \K  \otimes \Cliff{k}}(X)$, its class  in $\hE^1_D(X)$ is denoted by $[B]$. Let $\bar{k}\in \Z/2$ denote the $\mathrm{mod}$ 2 reduction of $k$.
For $A\in \mathscr{C}_{M_n(D)\otimes \Cliff{k}}(X)$ we show that the map $A\mapsto (\bar{k},[A\otimes \K])$ descends to an injective
group homomorphism $\hat{\theta}: \hat{Br}_D(X)\to H^0(X,\Z/2)\times  \hE^1_D(X)$ whose image we identify in the sequel. We will sometimes write $k_A$ for $k$ to trace this integer back to $A$.

Recall  from \cite{DP3} that the ungraded Brauer group $Br_D(X)$ consists of
equivalence classes of continuous fields $A\in \bigcup_{n\geq 1}\mathscr{C}_{M_n(D)}(X)$.
Two continuous fields $A_i\in \mathscr{C}_{M_{n_i}(D)}(X)$, $i=1,2$ define  the same class in $Br_D(X)$ if and only if
$$A_1 \otimes p_1C(X, M_{N_1}(D))p_1\cong A_2 \otimes p_2C(X, M_{N_2}(D))p_2 ,$$
for some full projections $p_i\in C(X,M_{N_i}(D)),$ $i=1,2$.
We have shown in \cite{DP3} that $\theta:Br_D(X)\to E^1_D(X)$, $[A]\mapsto [A\otimes \K]$, is an injective homomorphism onto the subgroup $\Tor \bar{E}^1_D(X)$ of $E^1_D(X)$. Recall that $\bar{E}^1_D(X)$ classifies $D\otimes K$-bundles with structure group $\Aut_0(D\otimes\K)$, the connected component of identity iof $\Aut(D\otimes\K)$.

\begin{theorem}\label{thm:Brauer_Serre}  Let $X$ be a finite connected CW-complex and let  $D$ be a stably finite strongly self-absorbing $C^*$-algebra satisfying the UCT, thus $D=\C$ or $D=\ZZ$ or $D=M_P$ for some set  of primes.
The map $$\hat{\theta}:\hat{Br}_D(X)\to H^0(X,\Z/2) \times H^1(X,\Z/2)\times_{_{tw}} \Tor{\bar{E}^1_D(X)},$$ $[A]_{\hat{Br}}\mapsto (\bar{k}_A, [A\otimes \K])$ is  an isomorphism of groups. The multiplication on $H^1(X,\Z/2)\times \Tor{\bar{E}^1_D(X)}$ is twisted as in Theorem~\ref{thm:basic}.
\end{theorem}

\begin{proof} If $C$ is negligible, then $C\otimes \K\cong C(X)\otimes \K \otimes D \otimes \Cliff{r}$ by the graded version of Brown's
theorem \cite{paperL.G.Brown.Stable.Isom}. Since the group operation on $\hE_D^1(X)$ coincides with the tensor product operation, if $A\in \mathscr{C}_{M_n(D)\otimes \Cliff{2k}}(X)$, then  $[A\otimes C \otimes \K]=[A\otimes \K]+[C \otimes \K]=[A\otimes \K]$ in $\hE_D^1(X)$. This implies that the map $\hat{\theta}: \hat{Br}_D(X)\to H^0(X,\Z/2)\times  \hE^1_D(X)$, $A\mapsto (\bar{k}_A,[A\otimes \K])$, is well-defined.
We describe next the image of this map.

 Consider the following commutative diagram
\[
	\xymatrix{
		Br_D(X)\ar[r]\ar[d]_{\theta} & \hat{Br}_D(X)\ar[d]^{\hat{\theta}_1}&\\
		E^1_D(X)\ar[r]^{i}&\hE^1_D(X)\ar[r]^-{\Phi} & H^1(X,\Z/2)\times E_D^1(X)
	}
\]
where $\hat{\theta}_1[A]_{\hat{Br}}=[A\otimes \K]$.
We have seen in equation~\eqref{eqn:splito} from the proof of Theorem~\ref{thm:Z2} that there is a bijective map
$\hE^1_{D}(X)\to H^1(X, \Z/2)\times E^1_{D}(X)$. We denote this bijection by~$\Phi$ and write its components as
$\Phi=(\delta_0,\varphi)$. We described the homomorphism $\delta_0$ in Theorem~\ref{thm:Z2}.
The set-theoretic map $\varphi:\hE^1_{D}(X) \to E^1_{D}(X)$ is induced by the map $\hat{\mathscr{C}}_{ D \otimes \K  \otimes \Cliff{2k}}(X) \to {\mathscr{C}}_{ D \otimes \K  \otimes M_{2^k}}(X)$ which identifies $\Cliff{2k}$ with $M_{2^k}$ as complex algebras and forgets the grading.
Now if  $A\in \hat{\mathscr{C}}_{M_n(D)\otimes \Cliff{2k}}(X)$, then $(\varphi\circ \hat{\theta}_1)[A]=(i\circ \theta)[\bar{A}]$, where
$\bar{A}\in \mathscr{C}_{M_n(D)\otimes M_{2^k}}(X)$ is $A$ regarded as an ungraded $C^*$-algebra.
This shows  that the image of $\hat{\theta}_1$ is contained in the subgroup $H^1(X,\Z/2)\times_{_{tw}} \Tor{\bar{E}^1_D(X)}$ of $\hat{E}^1_D(X)$ and hence the image of $\hat{\theta}$ is contained in $H^0(X,\Z/2) \times H^1(X,\Z/2)\times_{_{tw}} \Tor{\bar{E}^1_D(X)}.$
Let us show that this is precisely the image of $\hat{\theta}$. This  clearly reduces to identifying the image of $\hat{\theta}_1$.
Let $(x^0,x)\in H^1(X,\Z/2)\times \Tor{\bar{E}^1_D(X)}$. By \cite{DP3} there is some $B\in \mathscr{C}_{M_{n}(D)}(X)$
such that $(i\circ \theta )[B]_{Br}=x$. As explained in the proof of Proposition~\ref{prop:Z2}, there is a real line bundle $L$ on $X$
such that $\delta_0[\Cliff{L}]=x_0$. It follows that $(\Phi\circ \hat{\theta}_1)[B\otimes\Cliff{L}]=(x_0,x)$.
It remains to show that the monoid  $\hat{Br}(X)$ is a group and that $\hat{\theta}$ is injective. As noted in the proof of Theorem~2.15 from \cite{DP3}, if we show that $\hat{\theta}^{-1}(0)=[0]_{\hat{Br}}$, this property will imply not only that $\hat{\theta}$ is injective but also that $\hat{Br}(X)$ is a group. Let $A\in \hat{\mathscr{C}}_{M_n(D)\otimes \Cliff{r}}(X)$ and suppose that $\hat{\theta}[A]_{\hat{Br}}=0$. Thus
$r$ must be even and $[A\otimes \K]=0$ in $\hE^1_D(X)$. It follows that  there is an isomorphism $\Psi: A\otimes \K \to C(X)\otimes D \otimes \Cliff{2}\otimes \K.$ If we set $p=\Psi(1_A \otimes e)$, where $e$ is a rank one projection in $\K$, then $p$ is a full projection since $1_A\otimes e$ is full.
 After conjugating $\Psi$ by a unitary in the multiplier algebra we may arrange that $p\in C(X)\otimes D \otimes \Cliff{2}\otimes M_N$ for some integer $N\geq 1$.
It follows that $A\cong p\large(C(X)\otimes D \otimes \Cliff{2}\otimes M_N \large)p$,  so that $A$ is negligible.
\end{proof}

\bibliographystyle{abbrv}

\smaller[1]

\end{document}